\documentclass[indent=latex,theme=print,review=false,fontsize=12pt]{acmhomework}

\newcommand{\hP}{\widehat{P}}

\title{A Toolkit for Structured Lifts}

\author{Krzysztof Kapulkin \and Yufeng Li}

\begin{abstract}
  We develop a general framework for working with structured lifting problems, establishing closure and uniqueness properties of their solutions.
  In a subsequent paper, we apply these results to axiomatize computation rules of cubical type theory.
\end{abstract}

\begin{document}

\maketitle

\section*{Introduction}

When working with categorical models of dependent type theory, one often models pattern matching using compatible choices of solutions for a certain class of lifting problems, with the canonical example being the $\MsJ$-eliminator for identity types.
This is true verbatim for models such as contextual categories \cite{car86} and categories with families \cite{dyb96,awo18}, and is typically imposed on weaker models, like comprehension categories \cite{jac93}, type-theoretic fibration categories \cite{shu15}, and tribes \cite{joy17}, when describing their corresponding strict analogues.

Such compatible families of lifts arise naturally in the context of \emph{algebraic} weak factorization systems \cite{gt06,gar09} and algebraic model structures \cite{rie11}, a replacement of weak factorization systems and model structures that require that the factorizations and lifts not merely exist, but rather be chosen in a suitably compatible (hence algebraic) manner \cite{bg16}.
As a result, algebraic weak factorization systems have been used to supply models of dependent type theory \cite{bf22,awo+25}.
When working with the category of all models, however, their usefulness is limited due to the fact that the identity type factorization is in general not functorial \cite{gg08}.
In other words, the problem with using algebraic weak factorization systems for axiomatizing the structure present in the models is the ``factorization'' part of the weak factorization system.

One could in principle try to build a new algebraic weak factorization system on a model of dependent type theory, but two problems arise.
First, the standard tool for constructing algebraic weak factorization systems, the algebraic small object argument, requires the base category to admit certain colimits \cite{gar09}, which general models of type theory do not possess.
Second, even if it were possible to run the algebraic small object argument on the set of reflexivity maps, the resulting right class would not in general recover dependent projections.
All of this to say that algebraic weak factorization systems are simply too strong a framework to speak of models of dependent type theory.

In the present paper, we address these issues by developing the theory of structured lifts without the corresponding factorization requirements, and with minimal assumptions on the base category.
This is in contrast with more ``heavy duty'' conditions required to get the framework of algebraic weak factorization systems off the ground: for example, requiring that the two classes be given by (co-)algebras for a suitable (co-)monad on the base category.

We prove several closure properties of these lifts, for example, their compatibility with base change, composition, and Leibniz transpose.
We further investigate the uniqueness properties of these lifts.
All of these are done with an eye towards understanding computation rules of dependent type theory in general, and of cubical type theory in particular \cite{cchm15,ang+21}.
Indeed, for the three operations mentioned above: base change corresponds to the stability of eliminators under substitution, while composition and Leibniz transpose describe the compatibility between the filling operation and dependent sums and path types, respectively.
The uniqueness properties ensure that the computation rules hold in their expected (propositional) form.

By keeping track of exponentiable objects, all of our constructions can be formulated not only in the framework of universe category models of dependent type theory \cite{voe17,kl21}, but also in the framework of \citeauthor{uem23}'s categories with representable maps \cite{uem23}.

This paper is organized as follows.
In \cref{sec:struct-lift}, we introduce the requisite notions of lifts, structured lifts, and restricted lifts.
We also explain the relation between structured lifts and categorical models of type theory, formulated as Uemura's categories with representable maps.
In the subsequent four sections, we investigate different closure properties of structured lifts.
Specifically, in \cref{subsec:struct-lift-restrict}, we show that structured lifts can be restricted and composed; in \cref{subsec:struct-lift-rebase}, we prove their closure under base change; and in \cref{subsec:struct-ltrans}, we establish their closure under Leibniz transposes.
As a preliminary step towards \cref{subsec:struct-ltrans}, in \cref{sec:pushout-product-approx}, we design a way of speaking about pushout-product without appealing to any cocompleteness properties of the base category.
Finally, in \cref{sec:relating-lifts}, we describe a general framework for showing that two lifts are related, and we instantiate our results in a model category.


\section{Structured Lifts}\label{sec:struct-lift}
In this section, we will introduce the notion of uniform lifts.
At a high level, they are the uniform versions of lifting properties, where
the uniformity arise from functoriality of the product.
Throughout the rest of this paper, we fix a finitely complete category $\bC$
whose internal-Homs, when they exist, we denote by $[-,-]$ and whose
internal-Homs, when they exist, in each overslice $\sfrac{\bC}{C}$ we denote by
$[-,-]_C$.
The main example we have in mind are \citeauthor{uem23}'s categories with
representable maps (CwRs).

We first start by recalling what is a lifting problem and its associated
solution.
\begin{definition}
  A lifting problem of $U \to V$ against $E \to B$ is a pair of dashed maps
  $(u,v)$ as below making the square commute.
  A solution to the lifting problem $(u,v)$ is a diagonal filler $F$ to the
  square making the entire diagram commute.
  \begin{equation*}
    \begin{tikzcd}[cramped]
      U & E \\
      V & B
      \arrow[dashed, "u", from=1-1, to=1-2]
      \arrow[from=1-1, to=2-1]
      \arrow[from=1-2, to=2-2]
      \arrow[dotted, "F"{description}, from=2-1, to=1-2]
      \arrow[dashed, "v"', from=2-1, to=2-2]
    \end{tikzcd}
  \end{equation*}
\end{definition}

For ease of viewing, we color code by depicting lifting problems in
\textcolor{red0}{red}, the maps being lifted against each other in \textcolor{darkblue0}{blue}, and lifting solutions in \textcolor{yellow0}{green}.
Therefore, we would draw the above lifting problem and solution pair as
\begin{equation*}
  \begin{tikzcd}[cramped]
    U & E \\
    V & B
    \arrow[dashed, "u", color=red0, from=1-1, to=1-2]
    \arrow[from=1-1, to=2-1, color=darkblue0]
    \arrow[from=1-2, to=2-2, color=darkblue0]
    \arrow[dotted, "F"{description}, color=yellow0, from=2-1, to=1-2]
    \arrow[dashed, "v"', color=red0, from=2-1, to=2-2]
  \end{tikzcd}
\end{equation*}

Lifting solutions are dependent on the supplied lifting problem.
Therefore, given a family of lifting problems satisfying some form of
compatibility conditions, we can require the corresponding family of lifting
solutions to also be compatible.
In this paper, we are interested in the case when the compatibility conditions
arise from functoriality of the product.

\begin{definition}
  Fix maps $i \colon U \to V$ and $p \colon E \to B$.
  A \emph{family of lifts} is an association taking each object $X \in \bC$ and
  lifting problem $(u, v)$ of $X \times i$ against $p$ to a solution $F_X(u,v)$.
  \begin{equation*}
    \begin{tikzcd}[cramped]
      {X \times U} && E \\
      \\
      {X \times V} && B
      \arrow["u", color=red0, dashed, from=1-1, to=1-3]
      \arrow[draw=darkblue0, from=1-1, to=3-1]
      \arrow[draw=darkblue0, from=1-3, to=3-3]
      \arrow["{F_X(u,v)}"{description}, color=yellow0, dotted, from=3-1, to=1-3]
      \arrow["v"', color=red0, dashed, from=3-1, to=3-3]
    \end{tikzcd}
  \end{equation*}

  This family $F$ is said to be \emph{uniform} when one has
  $F_Y(u \cdot (t \times U), v \cdot (t \times V)) = F_X(u,v) \cdot (t \times
  V)$ for any $t \colon Y \to X$.
  \begin{equation*}
    \begin{tikzcd}[cramped]
      {Y \times U} &&& {X \times U} && E \\
      \\
      {Y \times V} &&& {X \times V} && B
      \arrow["{t \times U}", from=1-1, to=1-4]
      \arrow[color=darkblue0, from=1-1, to=3-1]
      \arrow["u", dashed, color=red0, from=1-4, to=1-6]
      \arrow[color=darkblue0, from=1-4, to=3-4]
      \arrow[color=darkblue0, from=1-6, to=3-6]
      \arrow["{F_Y(u \cdot (t \times U), v \cdot (t \times V))}"{description, pos=0.3}, color=yellow0, dotted, from=3-1, to=1-6]
      \arrow["{t \times V}"', from=3-1, to=3-4]
      \arrow["{F_X(u,v)}"{description}, color=yellow0, dotted, from=3-4, to=1-6]
      \arrow["v"', dashed, color=red0, from=3-4, to=3-6]
    \end{tikzcd}
  \end{equation*}
\end{definition}

In the above definitions, the dashed horizontal maps $u,v$ can be arbitrary.
However, we can also require that they factor through certain specific maps.
This gives rise to the notion of \emph{left-restricted} and
\emph{right-restricted} lifts and their uniform versions.

\begin{definition}
  A lifting problem of $U \to V$ against $E \to B$ \emph{left-restricted} along
  $V \to V'$ is a pair of dashed maps $(u,v')$ as below making the square on the
  left commute.
  A solution to this left-restricted lifting problem $(u,v)$ is a diagonal
  filler $F$ to the square making the entire diagram on the left commute.

  Dually, lifting problem of $U \to V$ against $E \to B$ \emph{right-restricted} along
  $B' \to B$ is a pair of dashed maps $(u,v)$ as above making the
  square on the right commute.
  A solution to this right-restricted lifting problem $(u,v)$ is a diagonal
  filler $F$ to the square making the entire diagram on the right commute.

  \begin{center}
    \begin{minipage}{0.45\linewidth}
      \begin{equation*}
        \begin{tikzcd}[cramped]
          U && E \\
          V & {V'} & B
          \arrow["u", color=red0, dashed, from=1-1, to=1-3]
          \arrow[color=darkblue0, from=1-1, to=2-1]
          \arrow[color=darkblue0, from=1-3, to=2-3]
          \arrow["F"{description}, color=yellow0, dotted, from=2-1, to=1-3]
          \arrow[color=red0, from=2-1, to=2-2]
          \arrow["{v'}"', color=red0, dashed, from=2-2, to=2-3]
        \end{tikzcd}
      \end{equation*}
    \end{minipage}
    \begin{minipage}{0.45\linewidth}
      \begin{equation*}
        \begin{tikzcd}[cramped]
          U && E \\
          V & {B'} & B
          \arrow["u", color=red0, dashed, from=1-1, to=1-3]
          \arrow[color=darkblue0, from=1-1, to=2-1]
          \arrow[color=darkblue0, from=1-3, to=2-3]
          \arrow["F"{description}, color=yellow0, dotted, from=2-1, to=1-3]
          \arrow["v"', color=red0, dashed, from=2-1, to=2-2]
          \arrow[color=red0, from=2-2, to=2-3]
        \end{tikzcd}
      \end{equation*}
    \end{minipage}
  \end{center}
\end{definition}

Parameterising over the choice of restricted lifting problems, we can similarly
define families of their lifting solutions and a corresponding uniformity
condition.

\begin{definition}\label{def:unif-lift}
  Fix maps $i \colon U \to V$ and $p \colon E \to B$.

  A family of lifts \emph{left-restricted} along $j \colon V \to V'$ is an
  association taking each object $X \in \bC$ and lifting problem $(u,v')$ of
  $X \times i$ left-restricted along $X \times j$ against $p$ to a solution
  $F_X(u,v')$.
  This family $F$ is said to be \emph{uniform} when one has one has
  $F_Y(u \cdot (t \times U), v' \cdot (t \times V')) = F_X(u,v') \cdot (t \times
  V)$ for any $t \colon Y \to X$.
  \begin{center}
    \begin{minipage}{0.45\linewidth}
      \begin{equation*}\small
        \begin{tikzcd}[cramped, column sep=small]
          {Y \times U} \\
          & {X \times U} &&&& E \\
          {Y \times V} && {Y \times V'} \\
          & {X \times V} && {X \times V'} && B
          \arrow[from=1-1, to=2-2]
          \arrow[color=red0, dashed, from=1-1, to=2-6]
          \arrow[from=1-1, to=3-1]
          \arrow["u"{description}, color=red0, dashed, from=2-2, to=2-6]
          \arrow[color=darkblue0, from=2-6, to=4-6]
          \arrow[color=yellow0, dotted, from=3-1, to=2-6]
          \arrow[color=red0, from=3-1, to=3-3]
          \arrow[from=3-1, to=4-2]
          \arrow[from=3-3, to=4-4]
          \arrow[color=red0, dashed, from=3-3, to=4-6]
          \arrow[color=yellow0, dotted, from=4-2, to=2-6]
          \arrow[color=red0, from=4-2, to=4-4]
          \arrow["v'"', color=red0, dashed, from=4-4, to=4-6]
          \arrow[color=darkblue0, from=2-2, to=4-2, crossing over]
        \end{tikzcd}
      \end{equation*}
    \end{minipage}
    \begin{minipage}{0.45\linewidth}
      \begin{equation*}
        \small
        \begin{tikzcd}[cramped, column sep=small]
          {Y \times U} \\
          & {X \times U} &&& E \\
          {Y \times V} \\
          & {X \times V} && {B'} & B
          \arrow[from=1-1, to=2-2]
          \arrow[color=red0, dashed, from=1-1, to=2-5]
          \arrow[color=darkblue0, from=1-1, to=3-1]
          \arrow["u"{description}, color=red0, dashed, from=2-2, to=2-5]
          \arrow[color=darkblue0, from=2-5, to=4-5]
          \arrow[color=yellow0, dotted, from=3-1, to=2-5]
          \arrow[from=3-1, to=4-2]
          \arrow[color=red0, dashed, from=3-1, to=4-4]
          \arrow[color=yellow0, dotted, from=4-2, to=2-5]
          \arrow["v"', color=red0, dashed, from=4-2, to=4-4]
          \arrow[color=red0, from=4-4, to=4-5]
          \arrow[color=darkblue0, from=2-2, to=4-2, crossing over]
        \end{tikzcd}
      \end{equation*}
    \end{minipage}
  \end{center}

  Dually, a family of lifts \emph{right-restricted} along $q \colon B' \to B$ is
  an association taking each object $X \in \bC$ and lifting problem $(u,v)$ of
  $X \times i$ right-restricted along $q$ against $p$ to a solution $F_X(u,v)$.
  This family is said to be \emph{uniform} when one has
  $F_Y(u \cdot (t \times U), v \cdot (t \times V)) = F_X(u,v) \cdot (t \times
  V)$ for any $t \colon Y \to X$.

  The set of all such left- and right-restricted uniform family of lifts, which
  we refer to as \emph{lifting structures} are respectively denoted by
  \begin{center}
    \begin{minipage}{0.45\linewidth}
      \begin{equation*}
        {\begin{tikzcd}[cramped, column sep=small, row sep=small]
            U \ar[d] \\ V \ar[r] & V'
          \end{tikzcd}}
        \squareslash
        {\begin{tikzcd}[cramped, column sep=small, row sep=small]
            E \ar[d] \\ B
          \end{tikzcd}}
      \end{equation*}
    \end{minipage}
    \begin{minipage}{0.45\linewidth}
      \begin{equation*}
        {\begin{tikzcd}[cramped, column sep=small, row sep=small]
            U \ar[d] \\ V
          \end{tikzcd}}
        \squareslash
        {\begin{tikzcd}[cramped, column sep=small, row sep=small]
            & E \ar[d] \\ B \ar[r] & B
          \end{tikzcd}}
      \end{equation*}
    \end{minipage}
  \end{center}
\end{definition}

\newsavebox{\dA}\newsavebox{\dAp}\newsavebox{\uB}\newsavebox{\uBr}\newsavebox{\uBp}\newsavebox{\uBpr}
\newsavebox{\Abot}\newsavebox{\Btop}\newsavebox{\dAr}

\begin{remark}
  The uniformity conditions of \Cref{def:unif-lift} can be phrased in terms of
  naturality conditions.

  \begin{lrbox}{\dA}{\scriptsize\begin{tikzcd}[column sep=small, row sep=small]
        \textcolor{darkblue0}{- \times U} \ar[d, color=darkblue0] \\
        \textcolor{darkblue0}{- \times V} &
        {\phantom{- \times V'}} \end{tikzcd}}
  \end{lrbox}
  \begin{lrbox}{\dAp}{\scriptsize\begin{tikzcd}[column sep=small, row sep=small]
        \textcolor{red0}{- \times U} \ar[d, color=red0] \\
        \textcolor{red0}{- \times V} \ar[r, color=red0] &
        \textcolor{red0}{- \times V'} \end{tikzcd}}
  \end{lrbox}
  \begin{lrbox}{\uB}{\scriptsize\begin{tikzcd}[column sep=small, row sep=small]
        \textcolor{darkblue0}{E} \ar[d, color=darkblue0] \\
        \textcolor{darkblue0}{B} \end{tikzcd}}
  \end{lrbox}
  \begin{lrbox}{\uBr}{\scriptsize\begin{tikzcd}[column sep=small, row sep=small]
        \textcolor{red0}{E} \ar[d, color=red0] \\
        \textcolor{red0}{B} \end{tikzcd}}
  \end{lrbox}
  \begin{lrbox}{\Abot}{\scriptsize\begin{tikzcd}[column sep=small, row sep=small]
        {\phantom{- \times \partial A}} \\ \textcolor{yellow0}{- \times V} & {\phantom{- \times A'}} \end{tikzcd}}
  \end{lrbox}
  \begin{lrbox}{\Btop}{\scriptsize\begin{tikzcd}[column sep=small, row sep=small]
        \textcolor{yellow0}{E} \\ \phantom{\underline{B}} \end{tikzcd}}
  \end{lrbox}

  We first consider the left-restricted uniform case with a map $U \to V$ on the
  left restricted along $V \to V'$ against $E \to B$ on the right.
  For each $X \in \bC$, the set of all unrestricted lifting problems of
  $X \times U \to X \times V$ against $E \to B$ is observed to be the Hom-set of
  the arrow category $\bC^\to(X \times U \to X \times V, E \to B)$.
  Letting $X$ vary, we obtain the presheaf of parameterised lifting problems.
  \begin{equation*}
    \textcolor{darkblue0}{\bC^\to\left(
        \usebox{\dA},
        \usebox{\uB}
      \right)}
  \end{equation*}
  Similarly, the set of lifting problems restricted along
  $X \times V \to X \times V'$ is observed to be the Hom-set
  $\bC^\to(X \times U \to X \times V \to X \times V', E \to B)$ and letting $X$
  vary then gives rise to the presheaf of parameterised restricted lifting
  problems.
  \begin{equation*}
    \textcolor{red0}{\bC^\to\left(
        \usebox{\dAp},
        \usebox{\uBr}
      \right)}
  \end{equation*}

  Given a restricted lifting problem
  $(u \colon X \times U \to E, v' \colon X \times V' \to B) \in \bC^\to(X \times
  U \to X \times V \to X \times V', E \to B)$, we can always forget the
  restriction and produce an unrestricted lifting problem
  $(u \colon X \times U \to E, v' \colon X \times V' \to B) \in
  \bC^\to(X \times U \to X \times V, E \to B)$.
  This operation is observed to be natural in $X$ so once again letting $X$ vary
  gives rise to a map between the two aforementioned presheaves.
  \begin{equation*}
    \begin{tikzcd}
      \textcolor{red0}{\bC^\to\left( \usebox{\dAp}, \usebox{\uBr} \right)}
      \ar[r, color=red0]
      &
      \textcolor{darkblue0}{\bC^\to\left( \usebox{\dA}, \usebox{\uB} \right)}
    \end{tikzcd}
  \end{equation*}

  On the other hand, a solution to each restricted lifting problem
  $(u \colon X \times U \to E, v' \colon X \times V' \to B) \in \bC^\to(X \times
  U \to X \times V \to X \times V', E \to B)$ is some choice of diagonal map
  $X \times V \to E$.
  Letting $X$ vary, the set of parameterised solution candidates is given by the
  presheaf
  \begin{equation*}
    \begin{tikzcd}[column sep=small]
      \textcolor{yellow0}{\bC\left( \usebox{\Abot}, \usebox{\Btop} \right)}
    \end{tikzcd}
  \end{equation*}
  Given such a diagonal map $X \times V \to E$, we can always recover the
  lifting problem it solves by pre-composing with $X \times U \to X \times V$
  and post-composing with $E \to B$ to obtain a map from
  $X \times U \to X \times V$ to $E \to B$.
  Therefore, we obtain a map taking each family of lifting solution candidate to
  the corresponding family of lifting problems it tries to solve.
  \begin{equation*}
    \begin{tikzcd}[column sep=small]
      \textcolor{yellow0}{\bC\left( \usebox{\Abot}, \usebox{\Btop} \right)}
      \ar[r, color=yellow0]
      &
      \textcolor{darkblue0}{\bC^\to\left( \usebox{\dA}, \usebox{\uB} \right)}
    \end{tikzcd}
  \end{equation*}

  So, to say that a lifting solution candidate indeed solves the lifting problem
  which it is assigned to, one requires a horizontal dotted map in the
  following.
  Therefore, the set of uniform left-restricted lifts is just the set of natural
  transformations in the folplowing slice category in $\widehat{\bC}$
  \begin{equation*}
    \begin{tikzcd}[column sep=small]
      \textcolor{red0}{\bC^\to\left( \usebox{\dAp}, \usebox{\uBr} \right)}
      \ar[rd, color=red0]
      \ar[rr, dotted]
      &
      &
      \textcolor{yellow0}{\bC\left( \usebox{\Abot}, \usebox{\Btop} \right)}
      \ar[ld, color=yellow0]
      \\
      &
      \textcolor{darkblue0}{\bC^\to\left( \usebox{\dA}, \usebox{\uB} \right)}
    \end{tikzcd}
  \end{equation*}

  Dually, given $B' \to B$, denote the set of right-restricted families of
  lifting problems as the presheaf
  \begin{equation*}
    {\color{red0}
      \bC^\to\left(
        {\scriptsize\begin{tikzcd}[cramped, column sep=small, row sep=small]
            - \times U \ar[d] \\ - \times V
          \end{tikzcd}},
        {\scriptsize\begin{tikzcd}[cramped, column sep=small, row sep=small]
            & E \ar[d] \\ B' \ar[r] & B
          \end{tikzcd}}
      \right)}
    \coloneqq
    \left(
    X \mapsto
    \left\{
      \begin{pmatrix}
        \begin{tikzcd}[cramped] X \times U \ar[r, "u", dashed, color=red0] & E \end{tikzcd}
        \\
        \begin{tikzcd}[cramped] X \times V \ar[r, "v"', dashed, color=red0] & B' \end{tikzcd}
      \end{pmatrix}
      ~\middle|~
      \begin{tikzcd}[column sep=small]
        X \times U
        \ar[rr, "u", dashed, color=red0]
        \ar[d, darkblue0] & &
        E
        \ar[d, darkblue0]
        \\
        X \times V
        \ar[r, "v"', dashed, color=red0]
        &
        B' \ar[r, color=red0]
        &
        B
      \end{tikzcd}
    \right\}\right) \in \widehat{\bC}
  \end{equation*}
  Then, the set of uniform right-restricted lifts along $B' \to B$ is the set of
  natural transformations in the following slice category in $\widehat{\bC}$.
  \begin{lrbox}{\dA}{\scriptsize\begin{tikzcd}[column sep=small, row sep=small]
        \textcolor{darkblue0}{- \times U} \ar[d, color=darkblue0] \\
        \textcolor{darkblue0}{- \times V} \end{tikzcd}}
  \end{lrbox}
  \begin{lrbox}{\dAr}{\scriptsize\begin{tikzcd}[column sep=small, row sep=small]
        \textcolor{red0}{- \times U} \ar[d, color=red0] \\
        \textcolor{red0}{- \times V} \end{tikzcd}}
  \end{lrbox}
  \begin{lrbox}{\uBpr}{\scriptsize\begin{tikzcd}[column sep=small, row sep=small]
        & \textcolor{red0}{E} \ar[d, color=red0] \\
        \textcolor{red0}{B'} \ar[r, color=red0]
        &
        \textcolor{red0}{B} \end{tikzcd}}
  \end{lrbox}
  \begin{lrbox}{\uBp}{\scriptsize\begin{tikzcd}[column sep=small, row sep=small]
        & \textcolor{darkblue0}{E} \ar[d, color=darkblue0] \\
        \phantom{B'} & \textcolor{darkblue0}{B} \end{tikzcd}}
  \end{lrbox}
  \begin{lrbox}{\Abot}{\scriptsize\begin{tikzcd}[column sep=small, row sep=small]
        {\phantom{- \times \partial A}} \\ \textcolor{yellow0}{- \times V} \end{tikzcd}}
  \end{lrbox}
  \begin{lrbox}{\Btop}{\scriptsize\begin{tikzcd}[column sep=small, row sep=small]
        & \textcolor{yellow0}{E} \\ \phantom{\underline{B'}} & \phantom{B} \end{tikzcd}}
  \end{lrbox}
  \begin{equation*}
    \begin{tikzcd}[column sep=small]
      \textcolor{red0}{\bC^\to\left( \usebox{\dAr}, \usebox{\uBpr} \right)}
      \ar[rd, color=red0]
      \ar[rr, dotted]
      &
      &
      \textcolor{yellow0}{\bC^\to\left( \usebox{\Abot}, \usebox{\Btop} \right)}
      \ar[ld, color=yellow0]
      \\
      &
      \textcolor{darkblue0}{\bC^\to\left( \usebox{\dA}, \usebox{\uBp} \right)}
    \end{tikzcd}
  \end{equation*}
\end{remark}

We observe that because $\bC$ is locally cartesian closed, the set of uniform
(restricted) lifts can be internalised using representability arguments as in
the following \Cref{lem:stable-left-rep,lem:stable-right-rep}.
\begin{lemma}\label{lem:stable-left-rep}
  Suppose one has maps $U \to V$ and $E \to B$ and $V \to V'$ between
  exponential objects.
  Then, one has representability of:
  \begin{lrbox}{\dA}{\scriptsize\begin{tikzcd}[column sep=small, row sep=small]
        \textcolor{darkblue0}{- \times U} \ar[d, color=darkblue0] \\
        \textcolor{darkblue0}{- \times V} &
        {\phantom{- \times V'}} \end{tikzcd}}
  \end{lrbox}
  \begin{lrbox}{\dAp}{\scriptsize\begin{tikzcd}[column sep=small, row sep=small]
        \textcolor{red0}{- \times U} \ar[d, color=red0] \\
        \textcolor{red0}{- \times V} \ar[r, color=red0] &
        \textcolor{red0}{- \times V'} \end{tikzcd}}
  \end{lrbox}
  \begin{lrbox}{\uB}{\scriptsize\begin{tikzcd}[column sep=small, row sep=small]
        \textcolor{darkblue0}{E} \ar[d, color=darkblue0] \\
        \textcolor{darkblue0}{B} \end{tikzcd}}
  \end{lrbox}
  \begin{lrbox}{\uBr}{\scriptsize\begin{tikzcd}[column sep=small, row sep=small]
        \textcolor{red0}{E} \ar[d, color=red0] \\
        \textcolor{red0}{B} \end{tikzcd}}
  \end{lrbox}
  \begin{lrbox}{\Abot}{\scriptsize\begin{tikzcd}[column sep=small, row sep=small]
        {\phantom{- \times \partial A}} \\ \textcolor{yellow0}{- \times V} & {\phantom{- \times A'}} \end{tikzcd}}
  \end{lrbox}
  \begin{lrbox}{\Btop}{\scriptsize\begin{tikzcd}[column sep=small, row sep=small]
        \textcolor{yellow0}{E} \\ \phantom{\underline{B}} \end{tikzcd}}
  \end{lrbox}
  \begin{enumerate}
    \item \emph{Restricted lifting problems.}
    \begin{equation*}
      \begin{tikzcd}
        \bC(-, \textcolor{red0}{[U, E] \times_{[U, B]} [V', B]})
        \ar[r, tail reversed, "\cong"{description}]
        & \textcolor{red0}{\bC^\to\left( \usebox{\dAp}, \usebox{\uBr} \right)}
      \end{tikzcd}
    \end{equation*}
    \item \emph{Unrestricted lifting problems.}
    \begin{equation*}
      \begin{tikzcd}
        \bC(-, \textcolor{darkblue0}{[U, E] \times_{[U, B]} [V, B]})
        \ar[r, tail reversed, "\cong"{description}]
        & \textcolor{darkblue0}{\bC^\to\left( \usebox{\dA}, \usebox{\uB} \right)}
      \end{tikzcd}
    \end{equation*}
    \item \emph{Lifting solutions.}
    \begin{equation*}
      \begin{tikzcd}
        \bC(-, \textcolor{yellow0}{[V,E]})
        \ar[r, tail reversed, "\cong"{description}]
        &
        \textcolor{yellow0}{\bC\left( \usebox{\Abot}, \usebox{\Btop} \right)}
      \end{tikzcd}
    \end{equation*}
  \end{enumerate}
\end{lemma}
\begin{proof}
  By continuity of the Yoneda embedding,
  \begin{align*}
    \bC(-, [U, E] \times_{[U, B]} [V', B])
    &\cong
    \bC(-, [U, E]) \times_{\bC(-, [U, B])} \bC(-, [V', B]) \\
    &\cong
      \bC(- \times U, E) \times_{\bC(- \times U, B)} \bC(- \times V', B)
  \end{align*}
  This is exactly the commutative squares from
  $- \times U \to - \times V'$ to $E \to B$, so one has the
  isomorphism in the first item.
  A similar argument gives the isomorphism in the second and third items.
\end{proof}
\begin{corollary}\label{cor:stable-left-rep}
  Suppose one has maps $U \to V$ and $E \to B$ and $V \to V'$ where $U,V,V'$ are
  exponentiable objects.
  Then, the set of left-restricted uniform lifts is just the Hom-set in the
  slice category.
  \begin{equation*}
    {\scriptsize\begin{tikzcd}[cramped, column sep=small, row sep=small]
        U \ar[d] \\ V \ar[r] & V'
      \end{tikzcd}}
    \squareslash
    {\scriptsize\begin{tikzcd}[cramped, column sep=small, row sep=small]
        E \ar[d] \\ B
      \end{tikzcd}} \cong
    \begin{tikzcd}
      \sfrac{\bC}{\textcolor{darkblue0}{[U, E] \times_{[U, B]} [V, B]}}
      (\textcolor{red0}{[U, E] \times_{[U, B]} [V', B]},
      \textcolor{yellow0}{[V,E]})
    \end{tikzcd}
  \end{equation*}
\end{corollary}
\begin{proof}
  Immediately by \Cref{lem:stable-left-rep}.
\end{proof}

Dually, we have an internalisation result on the right.
\begin{lemma}\label{lem:stable-right-rep}
  Suppose one has maps $U \to V$ and $E \to B$ and $B' \to B$ where $U,V$ are
  exponentiable objects.
  Then, one has representability of:
  \begin{lrbox}{\dA}{\scriptsize\begin{tikzcd}[column sep=small, row sep=small]
        \textcolor{darkblue0}{- \times U} \ar[d, color=darkblue0] \\
        \textcolor{darkblue0}{- \times V} \end{tikzcd}}
  \end{lrbox}
  \begin{lrbox}{\dAr}{\scriptsize\begin{tikzcd}[column sep=small, row sep=small]
        \textcolor{red0}{- \times U} \ar[d, color=red0] \\
        \textcolor{red0}{- \times V} \end{tikzcd}}
  \end{lrbox}
  \begin{lrbox}{\uBpr}{\scriptsize\begin{tikzcd}[column sep=small, row sep=small]
        & \textcolor{red0}{E} \ar[d, color=red0] \\
        \textcolor{red0}{B'} \ar[r, color=red0]
        &
        \textcolor{red0}{B} \end{tikzcd}}
  \end{lrbox}
  \begin{lrbox}{\uBp}{\scriptsize\begin{tikzcd}[column sep=small, row sep=small]
        & \textcolor{darkblue0}{E} \ar[d, color=darkblue0] \\
        \phantom{B'} & \textcolor{darkblue0}{B} \end{tikzcd}}
  \end{lrbox}
  \begin{lrbox}{\Abot}{\scriptsize\begin{tikzcd}[column sep=small, row sep=small]
        {\phantom{- \times \partial A}} \\ \textcolor{yellow0}{- \times V} \end{tikzcd}}
  \end{lrbox}
  \begin{lrbox}{\Btop}{\scriptsize\begin{tikzcd}[column sep=small, row sep=small]
        & \textcolor{yellow0}{E} \\ \phantom{\underline{B'}} & \phantom{B} \end{tikzcd}}
  \end{lrbox}
  \begin{enumerate}
    \item \emph{Restricted lifting problems.}
    \begin{equation*}
      \begin{tikzcd}
        \bC(-, \textcolor{red0}{[U, E] \times_{[U, B]} [V, B']})
        \ar[r, tail reversed, "\cong"{description}]
        &
        \textcolor{red0}{\bC^\to\left( \usebox{\dAr}, \usebox{\uBpr} \right)}
      \end{tikzcd}
    \end{equation*}
    \item \emph{Unrestricted lifting problems.}
    \begin{equation*}
      \begin{tikzcd}
        \bC(-, \textcolor{blue0}{[U, E] \times_{[U, B]} [V, B]})
        \ar[r, tail reversed, "\cong"{description}]
        &
        \textcolor{darkblue0}{\bC^\to\left( \usebox{\dA}, \usebox{\uBp} \right)}
      \end{tikzcd}
    \end{equation*}
    \item \emph{Lifting solutions.}
    \begin{equation*}
      \begin{tikzcd}
        \bC(-, \textcolor{yellow0}{[V,E]})
        \ar[r, tail reversed, "\cong"{description}]
        &
        \textcolor{yellow0}{\bC^\to\left( \usebox{\Abot}, \usebox{\Btop} \right)}
      \end{tikzcd}
    \end{equation*}
  \end{enumerate}
\end{lemma}
\begin{proof}
  Identical argument by representability as in \Cref{lem:stable-left-rep}.
\end{proof}

\begin{corollary}\label{cor:stable-right-rep}
  Suppose one has maps $U \to V$ and $E \to B$ and $B' \to B$ where $U,V$ are
  exponentiable objects.
  Then, the set of left-restricted uniform lifts is just the Hom-set in the
  slice category.
  \begin{equation*}
    {\scriptsize\begin{tikzcd}[cramped, row sep=small, column sep=small]
        U \ar[d] \\ V
      \end{tikzcd}}
    \squareslash
    {\scriptsize\begin{tikzcd}[cramped, row sep=small, column sep=small]
        & E \ar[d] \\ B' \ar[r] & B
      \end{tikzcd}} \cong
    \begin{tikzcd}
      \sfrac{\bC}{\textcolor{blue0}{[U, E] \times_{[U, B]} [V, B]}}
      (\textcolor{red0}{[U, E] \times_{[U, B]} [V, B']}, \textcolor{yellow0}{[V,E]})
    \end{tikzcd}
  \end{equation*}
\end{corollary}

As an immediate consequence of internalisation, we obtain the following result on the existence of
stable lifting structures.
\begin{theorem}\label{thm:struct-lift-sect}
  Let $\bC$ a locally cartesian closed category model category in which the
  pullback-power of (trivial cofibration, fibration)- or (cofibration, trivial
  fibration)-pairs admit sections.
  Then, the sets of structured lifts
  \begin{center}
    \begin{minipage}{0.45\linewidth}
      \begin{equation*}
        {\begin{tikzcd}[cramped, column sep=small, row sep=small]
            U \ar[d] \\ V \ar[r] & V'
          \end{tikzcd}}
        \squareslash
        {\begin{tikzcd}[cramped, column sep=small, row sep=small]
            E \ar[d] \\ B
          \end{tikzcd}}
      \end{equation*}
    \end{minipage}
    \begin{minipage}{0.45\linewidth}
      \begin{equation*}
        {\begin{tikzcd}[cramped, row sep=small, column sep=small]
            U \ar[d] \\ V
          \end{tikzcd}}
        \squareslash
        {\begin{tikzcd}[cramped, row sep=small, column sep=small]
            & E \ar[d] \\ B' \ar[r] & B
          \end{tikzcd}}
      \end{equation*}
    \end{minipage}
  \end{center}
  and are non-empty as long as $(U \to V, E \to B)$ is either a (trivial
  cofibration, fibration)- or a (cofibration, trivial fibration)-pair.
\end{theorem}
\begin{proof}
  By \Cref{lem:stable-left-rep,lem:stable-right-rep}.
\end{proof}

\subsection*{Lifting Structures in CwRs}
In \citeauthor{uem23}'s framework of categories with representable maps (CwRs),
a type theory is \emph{defined} to be a CwR \cite[Definition 3.2.3]{uem23}.
The internalisation results of \Cref{cor:stable-left-rep,cor:stable-right-rep}
now allows one to freely extend a CwR $\bT$ with a lifting structure.

In order to state this construction, we recall some preliminaries about CwRs
from \cite{uem23,jel25}.

\begin{definition}[{\cite[Definition 2.3.1]{jel25}, cf.~\cite[Definition
    3.2.1]{uem23}}]\label{def:cwr}
  A \emph{category with representable maps (CwR)} is a category $\bC$ with
  finite limits equipped with a replete wide subcategory of pullback-stable
  class $\MsR_\bC$ of exponentiable maps called the \emph{representable maps}.

  A map of CwRs is a map between their underlying categories finite limits,
  representable maps, and pushforwards along representable maps.

  Denote by $\CwR$ the 2-category of (small) CwRs, maps of CwRs, and natural
  transformations whose square at representable maps are pullbacks.
\end{definition}

\begin{definition}\label{def:marked-cat}
  Denote by $\Cat_\Msm$ the 2-category of (small) categories with marked maps
  and squares.

  That is, a marked category is a category $\bC$ equipped with two specific
  choices of replete wide subcategories $\McM_\bC \hookrightarrow \bC$ and
  $\McS_\bC \hookrightarrow \bC^\to$.
  A 1-cell between marked categories is exactly a functor between underlying
  categories sending marked maps and squares to marked maps and squares.
  A 2-cell between 1-cells of marked categories is a natural transformation
  whose naturality square at marked maps are pullbacks.
\end{definition}

The representable maps of a CwR gives it the structure of a marked category by
taking the marked maps as representable maps and marked squares as pullback
squares.
Conversely, each marked category freely gives rise to a CwR due to the following
result by \citeauthor{jel25}.

\begin{theorem}[{\cite[Corollaries 3.2.16 and 3.2.17]{jel25}}]
  The forgetful 2-functor $\abs{-} \colon \CwR \to \Cat_\Msm$ has a left
  biadjoint $\vbr{-} \colon \Cat_\Msm \to \CwR$ and $\CwR$ has all bicolimits.
  \def\endingmark{\qedsymbol}
\end{theorem}

We now use this result to freely extend CwRs with lifting structures, which
amounts to formally adjoining a map between objects in a slice category.

\begin{construction}\label{constr:cwr-free-quotient}
  Let $\bC$ be a CwR.
  Fix two maps $\bI \to \bJ$ and $\bI \to \abs{\bC}$ of
  marked categories.
  Then, we write $\bC \cup_\bI \bJ$ for the following bipushout in $\CwR$
  \begin{equation*}
    \begin{tikzcd}[cramped]
      {\vbr{\bI}} & \bC \\
      {\vbr{\bJ}} & {\bC \cup_\bI \bJ}
      \arrow[from=1-1, to=1-2]
      \arrow[from=1-1, to=2-1]
      \arrow[from=1-2, to=2-2]
      \arrow[from=2-1, to=2-2]
      \arrow["\lrcorner"{anchor=center, pos=0.15, scale=1.5, rotate=180}, draw=none, from=2-2, to=1-1]
    \end{tikzcd}
  \end{equation*}
  where the map $\vbr{\bI} \to \bC$ in the top row is the
  $(\vbr{-} \dashv \abs{-})$-transpose of the map $\bI \to \abs{\bC}$ and map in
  the left row is the image of $\bI \to \bJ$ under $\vbr{-}$, as per
  \cite[Corollaries 3.2.16 and 3.2.17]{jel25}.
\end{construction}

\begin{construction}\label{constr:free-lift-ext}
  Let $\bC$ be a CwR and fix maps $U \to V \to V'$ between exponential objects
  and $E \to B$ in $\bC$.
  Define the CwR obtained from $\bC$ by freely extending with a lifting
  structure of $U \to V$ against $E \to B$ restricted along $V \to V'$ to be the
  following CwR obtained by applying \Cref{constr:cwr-free-quotient}
  \begin{equation*}
    \bC\left[
      {\scriptsize\begin{tikzcd}[column sep=small, row sep=small]
          U \ar[d] \\ V \ar[r] & V'
        \end{tikzcd}}
      \squareslash
      {\scriptsize\begin{tikzcd}[column sep=small, row sep=small]
          E \ar[d] \\ B
        \end{tikzcd}}
    \right]
    \coloneqq \bC
    \cup_{\set{\textcolor{red0}{\bullet \to {}} \textcolor{darkblue0}{\bullet} \textcolor{yellow0}{{} \leftarrow \bullet}}}
    \left\{
      {\scriptsize\begin{tikzcd}[column sep=small, row sep=small]
        \textcolor{red0}{[U,E] \times_{[U,B]} [V',B]} \ar[d, color=red0]
        \ar[r]
        &
        \textcolor{yellow0}{[V,E]} \ar[ld, color=yellow0]
        \\
        \textcolor{darkblue0}{[U,E] \times_{[U,B]} [V,B]}
      \end{tikzcd}}
    \right\}
  \end{equation*}
  This bipushout freely adds a map
  $\textcolor{red0}{[U,E] \times_{[U,B]} [V',B]} \to \textcolor{yellow0}{[V,E]}$
  over $\textcolor{darkblue0}{[U,E] \times_{[U,B]} [V',B]}$ to $\bC$ in $\CwR$.

  Dually, if $U \to V$ is a map between exponentiable objects in $\bC$ and one
  has maps $E \to B$ along with $B' \to B$ then the CwR with a formal lifting
  structure of $U \to V$ against $E \to B$ restricted along $B' \to B$ is
  obtained by the bipushout in $\CwR$ obtained by \Cref{constr:cwr-free-quotient}.
  \begin{equation*}
    \bC\left[
      {\scriptsize\begin{tikzcd}[column sep=small, row sep=small]
          U \ar[d] \\ V
        \end{tikzcd}}
      \squareslash
      {\scriptsize\begin{tikzcd}[column sep=small, row sep=small]
          & E \ar[d] \\ B' \ar[r] & B
        \end{tikzcd}}
    \right]
    \coloneqq \bC
    \cup_{\set{\textcolor{red0}{\bullet \to {}} \textcolor{darkblue0}{\bullet} \textcolor{yellow0}{{} \leftarrow \bullet}}}
    \left\{
      {\scriptsize\begin{tikzcd}[column sep=small, row sep=small]
          \textcolor{red0}{[U,E] \times_{[U,B]} [V,B']} \ar[d, color=red0]
          \ar[r]
          &
          \textcolor{yellow0}{[V,E]} \ar[ld, color=yellow0]
          \\
          \textcolor{darkblue0}{[U,E] \times_{[U,B]} [V,B]}
        \end{tikzcd}}
    \right\}
  \end{equation*}
  \newsavebox{\UVVp}
  \begin{lrbox}{\UVVp}
    \scriptsize\begin{tikzcd}[column sep=small, row sep=small]
      U \ar[d] \\ V \ar[r] & V'
    \end{tikzcd}
  \end{lrbox}
  \newsavebox{\EB}
  \begin{lrbox}{\EB}
    \scriptsize\begin{tikzcd}[column sep=small, row sep=small]
      E \ar[d] \\ B
    \end{tikzcd}
  \end{lrbox}
  \newsavebox{\ResLP}
  \begin{lrbox}{\ResLP}
    \scriptsize\begin{tikzcd}[column sep=small, row sep=small]
      \textcolor{red0}{[U,E] \times_{[U,B]} [V',B]} \ar[d, color=red0]
      &
      \textcolor{yellow0}{[V,E]} \ar[ld, color=yellow0]
      \\
      \textcolor{darkblue0}{[U,E] \times_{[U,B]} [V,B]}
    \end{tikzcd}
  \end{lrbox}
  \newsavebox{\ResLPSol}
  \begin{lrbox}{\ResLPSol}
    \scriptsize\begin{tikzcd}[column sep=small, row sep=small]
      \textcolor{red0}{[U,E] \times_{[U,B]} [V',B]} \ar[d, color=red0]
      \ar[r]
      &
      \textcolor{yellow0}{[V,E]} \ar[ld, color=yellow0]
      \\
      \textcolor{darkblue0}{[U,E] \times_{[U,B]} [V,B]}
    \end{tikzcd}
  \end{lrbox}
  %
  %
  %
  \newsavebox{\UV}
  \begin{lrbox}{\UV}
    \scriptsize\begin{tikzcd}[column sep=small, row sep=small]
      U \ar[d] \\ V
    \end{tikzcd}
  \end{lrbox}
  \newsavebox{\EBBp}
  \begin{lrbox}{\EBBp}
    \scriptsize\begin{tikzcd}[column sep=small, row sep=small]
      & E \ar[d] \\ B' \ar[r] & B
    \end{tikzcd}
  \end{lrbox}
  \newsavebox{\RResLP}
  \begin{lrbox}{\RResLP}
    \scriptsize\begin{tikzcd}[column sep=small, row sep=small]
      \textcolor{red0}{[U,E] \times_{[U,B]} [V,B']} \ar[d, color=red0]
      &
      \textcolor{yellow0}{[V,E]} \ar[ld, color=yellow0]
      \\
      \textcolor{darkblue0}{[U,E] \times_{[U,B]} [V,B]}
    \end{tikzcd}
  \end{lrbox}
  \newsavebox{\RResLPSol}
  \begin{lrbox}{\RResLPSol}
    \scriptsize\begin{tikzcd}[column sep=small, row sep=small]
      \textcolor{red0}{[U,E] \times_{[U,B]} [V,B']} \ar[d, color=red0] \ar[r]
      &
      \textcolor{yellow0}{[V,E]} \ar[ld, color=yellow0]
      \\
      \textcolor{darkblue0}{[U,E] \times_{[U,B]} [V,B]}
    \end{tikzcd}
  \end{lrbox}
  %
  %
\end{construction}

\begin{theorem}\label{thm:free-lift-ext}
  Let $\bC$ be a CwR and fix maps $U \to V \to V'$ between representable objects
  (i.e. objects whose maps into the terminal object are representable) and
  $E \to B$ in $\bC$.
  Then, isomorphism classes of maps as on the left below correspond bijectively
  to isomorphism classes of maps $F \colon \bC \to \bD \in \CwR$ equipped with a
  choice of a lifting structure of $FU \to FV$ against $FE \to FB$ restricted
  along $FV \to FV'$.
  \begin{center}
    \begin{minipage}{0.45\linewidth}
      \begin{equation*}
        \bC\left[
          {\scriptsize\begin{tikzcd}[column sep=small, row sep=small]
              U \ar[d] \\ V \ar[r] & V'
            \end{tikzcd}}
          \squareslash
          {\scriptsize\begin{tikzcd}[column sep=small, row sep=small]
              E \ar[d] \\ B
            \end{tikzcd}}
        \right] \to \bD \in \CwR
      \end{equation*}
    \end{minipage}
    \begin{minipage}{0.45\linewidth}
      \begin{equation*}
        \bC\left[
          {\scriptsize\begin{tikzcd}[column sep=small, row sep=small]
              U \ar[d] \\ V \ar[r] & V'
            \end{tikzcd}}
          \squareslash
          {\scriptsize\begin{tikzcd}[column sep=small, row sep=small]
              E \ar[d] \\ B
            \end{tikzcd}}
        \right] \to \bD \in \CwR
      \end{equation*}
    \end{minipage}
  \end{center}
  Dually, given $B' \to B$ in $\bC$, isomorphism classes of maps as on the right
  above correspond to isomorphism classes of maps
  $F \colon \bC \to \bD \in \CwR$ equipped with a choice of a lifting structure
  of $FU \to FV$ against $FE \to FB$ restricted along $FB' \to FB$.
\end{theorem}
\begin{proof}
  We show the case for left restricted lifting structures.
  By the universal property of the bipushout, isomorphism classes of maps
  \begin{equation*}
    \bC
    \cup_{\set{\textcolor{red0}{\bullet \to {}} \textcolor{darkblue0}{\bullet} \textcolor{yellow0}{{} \leftarrow \bullet}}}
    \left\{
      {\scriptsize\begin{tikzcd}[column sep=small, row sep=small]
          \textcolor{red0}{[U,E] \times_{[U,B]} [V',B]} \ar[d, color=red0]
          \ar[r]
          &
          \textcolor{yellow0}{[V,E]} \ar[ld, color=yellow0]
          \\
          \textcolor{darkblue0}{[U,E] \times_{[U,B]} [V,B]}
        \end{tikzcd}}
    \right\}
    \to \bD \in \CwR
  \end{equation*}
  correspond uniquely to isomorphism classes of maps $F \colon \bC \to \bD$
  equipped with a choice of a map
  $F(\textcolor{red0}{[U,E] \times_{[U,B]} [V',B]}) \to
  F(\textcolor{yellow0}{[V,E]})$ over
  $F(\textcolor{darkblue0}{[U,E] \times_{[U,B]} [V,B]})$.
  But $F$ is left exact and preserves all pushforwards along representable maps,
  so $F$ sends the cospan
  \begin{equation*}
    \textcolor{red0}{[U,E] \times_{[U,B]} [V',B] \to {}}
    \textcolor{darkblue0}{[U,E] \times_{[U,B]} [V,B]}
    \textcolor{yellow0}{\leftarrow [V,E]} \in \bC
  \end{equation*}
  to
  \begin{equation*}
    \textcolor{red0}{[FU,FE] \times_{[FU,FB]} [FV',FB] \to {}}
    \textcolor{darkblue0}{[FU,FE] \times_{[FU,FB]} [FV,FB]}
    \textcolor{yellow0}{\leftarrow [FV,FE]} \in \bD
  \end{equation*}
  Because $F \colon \bC \to \bD$ is a map of CwRs and $U \to V \to V'$ are maps
  between representable objects the map $FU \to FV \to FV'$ is again a map
  between representable, and thus exponentiable, objects.
  Thus, the lifting structures of $FU \to FV$ against $FE \to FB$ restricted
  along $FV \to FV'$ are exactly maps
  $\textcolor{red0}{[FU,FE] \times_{[FU,FB]} [FV',FB]} \to
  \textcolor{yellow0}{[FV,FE]}$ over
  $\textcolor{darkblue0}{[FU,FE] \times_{[FU,FB]} [FV,FB]}$ by
  \Cref{cor:stable-left-rep}.
\end{proof}


\section{Restricting and Composing Structured Lifts}\label{subsec:struct-lift-restrict}
We now give various ways of restricting lifting structures to construct new
lifting structures from old.
We first fix the left map and put on various restrictions on the right map in
\Cref{constr:struct-lift-right-res,constr:struct-lift-right-pb}.
We also show that structured lift against composable pairs of right maps induce
structured lifts against the composite in \Cref{constr:struct-lift-right-comp}.
Then, we fix the right map and put on various restrictions on the left map in
\Cref{constr:struct-lift-left-res,constr:struct-lift-left-retract}.

First, we show that restricting an existing structured lift on the right gives a
restricted structured lift.
\begin{construction}\label{constr:struct-lift-right-res}
  Fix maps $U \to V$ and $E \to B$ and $q \colon B' \to B$.
  We define a map as on the left
  \begin{center}
    \begin{minipage}{0.45\linewidth}
      \begin{equation*}
        \left(
          \begin{tikzcd}[cramped, row sep=small, column sep=small]
            U \ar[d] \\ V
          \end{tikzcd}
          \squareslash
          \begin{tikzcd}[cramped, row sep=small, column sep=small]
            E \ar[d] \\ B
          \end{tikzcd}
        \right)
        \xrightarrow{\qquad}
        \left(
          \begin{tikzcd}[cramped, row sep=small, column sep=small]
            U \ar[d] \\ V
          \end{tikzcd}
          \squareslash
          \begin{tikzcd}[cramped, row sep=small, column sep=small]
            & E \ar[d] \\ B' \ar[r, "q"'] & B
          \end{tikzcd}
        \right)
      \end{equation*}
    \end{minipage}
    \begin{minipage}{0.45\linewidth}
      \begin{equation*}
        \begin{tikzcd}[cramped]
          {X \times U} && E \\
          {X \times V} & {B'} & {B}
          \arrow["u", dashed, from=1-1, to=1-3, color=red0]
          \arrow[from=1-1, to=2-1, color=darkblue0]
          \arrow[from=1-3, to=2-3, color=darkblue0]
          \arrow["{F(u, qv)}"{description}, dotted, from=2-1, to=1-3, color=yellow0]
          \arrow["{v}"', dashed, from=2-1, to=2-2, color=red0]
          \arrow["{q}"', from=2-2, to=2-3, color=red0]
        \end{tikzcd}
      \end{equation*}
    \end{minipage}
  \end{center}
  where each structured lift $F$ of $U \to V$ against $E \to B$ is taken to the
  structured lift that solves each lifting problem $(u, v)$ with the solution
  $F(u, qv)$.
\end{construction}

We next note that the restriction on the right can be removed simply by pulling
back the right map.
\begin{construction}\label{constr:struct-lift-right-pb}
  Suppose one has a map $U \to V$ along with a pullback $E' \to B'$ of $E \to B$
  along $q \colon B' \to B$.
  \begin{equation*}
    \begin{tikzcd}[cramped]
      E' \ar[d] \ar[r, "q'"] \ar[rd, draw=none, "{\lrcorner}"{pos=0.15,scale=1.5}]
      & E \ar[d]
      \\ B' \ar[r, "q"']
      & B
    \end{tikzcd}
  \end{equation*}
  Then, we define a bijection as follows
  \begin{equation*}
    \left(
      \begin{tikzcd}[cramped, row sep=small, column sep=small]
        U \ar[d] \\ V
      \end{tikzcd}
      \squareslash
      \begin{tikzcd}[cramped, row sep=small, column sep=small]
        & B \ar[d] \\ B' \ar[r, "q"'] & B
      \end{tikzcd}
    \right)
    \xrightarrow{\quad\cong\quad}
    \left(
      \begin{tikzcd}[cramped, row sep=small, column sep=small]
        U \ar[d] \\ V
      \end{tikzcd}
      \squareslash
      \begin{tikzcd}[cramped, row sep=small, column sep=small]
        E' \ar[d] \\ B'
      \end{tikzcd}
    \right)
  \end{equation*}
  It sends each restricted $F$ of $U \to V$ against $E \to B$ restricted along
  $B' \to B$ to the lifting structure which solves lifting problems $(u,v)$ with
  the solution $(v, F(q'u, qv))$.
  \begin{equation*}
    \begin{tikzcd}[cramped]
      {X \times U} && {E'} && E \\
      \\
      {X \times V} && {B'} && {B}
      \arrow["u", dashed, from=1-1, to=1-3, color=red0]
      \arrow["q'", from=1-3, to=1-5]
      \arrow["\lrcorner"{anchor=center, pos=0.05,scale=1.5}, draw=none, from=1-3, to=3-5]
      \arrow[from=1-1, to=3-1, color=darkblue0]
      \arrow[from=1-3, to=3-3, color=darkblue0]
      \arrow[from=1-5, to=3-5, color=darkblue0]
      \arrow["{v}"', dashed, from=3-1, to=3-3, color=red0]
      \arrow["{q}"', from=3-3, to=3-5]
      \arrow["{(v, F(q'u, qv))}"{description}, dotted, from=3-1, to=1-3, color=yellow0]
      \arrow["{F(q'u, qv)}"{description, pos=0.7}, curve={height=12pt}, dotted, from=3-1, to=1-5, crossing over, color=yellow0]
    \end{tikzcd}
  \end{equation*}
\end{construction}

In the next construction, we show how structured lifts against composable pairs
of right maps induce structured lifts against their composite.
\begin{construction}\label{constr:struct-lift-right-comp}
  Fix maps $U \to V$ and a composable pair of maps
  $E' \xrightarrow{p'} E \xrightarrow{p} B$.
  Then, we define a map
  \begin{equation*}
    \left(
      \begin{tikzcd}[cramped, row sep=small, column sep=small]
        U \ar[d] \\ V
      \end{tikzcd}
      \squareslash
      \begin{tikzcd}[cramped, row sep=small, column sep=small]
        E' \ar[d, "{p'}"] \\ E
      \end{tikzcd}
    \right)
    \times
    \left(
      \begin{tikzcd}[cramped, row sep=small, column sep=small]
        U \ar[d] \\ V
      \end{tikzcd}
      \squareslash
      \begin{tikzcd}[cramped, row sep=small, column sep=small]
        E \ar[d, "{p}"] \\ B
      \end{tikzcd}
    \right)
    \xrightarrow{\qquad}
    \left(
      \begin{tikzcd}[cramped, row sep=small, column sep=small]
        U \ar[d] \\ V
      \end{tikzcd}
      \squareslash
      \begin{tikzcd}[cramped, row sep=small, column sep=small]
        E' \ar[d, "{p'}"] \\ E \ar[d, "{p}"] \\ B
      \end{tikzcd}
    \right)
  \end{equation*}
  that sends a pair $F' \in (U \to V) \squareslash (E' \to E)$ and
  $F \in (U \to V) \squareslash (E \to B)$ to the lifting structure which solves
  lifting problems $(u,v)$ with the consecutive lift $F'(u, F(p'u, v))$.
  \begin{equation*}
    \begin{tikzcd}[cramped]
      {X \times U} && E' \\
      && {E} \\
      {X \times V} && {B}
      \arrow["u", dashed, from=1-1, to=1-3, color=red0]
      \arrow[from=1-1, to=3-1, color=darkblue0]
      \arrow["p'", from=1-3, to=2-3 , color=darkblue0]
      \arrow["{p}", from=2-3, to=3-3, color=darkblue0]
      \arrow["{F'(u, F(p'u,v))}"{description}, dotted, from=3-1, to=1-3, color=yellow0]
      \arrow["{F(p'u,v)}"{description}, dotted, from=3-1, to=2-3, color=yellow0]
      \arrow["{v}"', dashed, from=3-1, to=3-3, color=red0]
    \end{tikzcd}
  \end{equation*}
\end{construction}

We now fix the right map and put on various restrictions on the left map.
First, we show the left analogue of \Cref{constr:struct-lift-right-res}.
\begin{construction}\label{constr:struct-lift-left-res}
  Fix maps $U \to V$ and $i \colon V \to V'$ and $E \to B$.
  We construct a map as on the left
  \begin{center}
    \begin{minipage}{0.45\linewidth}
      \begin{equation*}
        \left(
          \begin{tikzcd}[cramped, row sep=small, column sep=small]
            U \ar[d] \\ V
          \end{tikzcd}
          \squareslash
          \begin{tikzcd}[cramped, row sep=small, column sep=small]
            E \ar[d] \\ B
          \end{tikzcd}
        \right)
        \xrightarrow{\qquad}
        \left(
          \begin{tikzcd}[cramped, row sep=small, column sep=small]
            U \ar[d] \\ V \ar[r, "i"'] & V'
          \end{tikzcd}
          \squareslash
          \begin{tikzcd}[cramped, row sep=small, column sep=small]
            E \ar[d] \\ B
          \end{tikzcd}
        \right)
      \end{equation*}
    \end{minipage}
    \begin{minipage}{0.45\linewidth}
      \begin{equation*}
        \begin{tikzcd}[cramped]
          {X \times U} && B \\
          {X \times V} & {X \times V'} & {\underline{B}}
          \arrow["u", color=red0, dashed, from=1-1, to=1-3]
          \arrow[from=1-1, to=2-1, color=darkblue0]
          \arrow[from=1-3, to=2-3, color=darkblue0]
          \arrow["{F(u, v' \cdot (X \times i))}"{description}, color=yellow0, dotted, from=2-1, to=1-3]
          \arrow["{X \times i}"', from=2-1, to=2-2, color=red0]
          \arrow["{v'}"', color=red0, dashed, from=2-2, to=2-3]
        \end{tikzcd}
      \end{equation*}
    \end{minipage}
  \end{center}
  where the image of $F \in (U \to V) \squareslash (E \to B)$ solves each
  lifting problem $(u,v')$ with the solution $F(u, v' \cdot (X \times i))$.
\end{construction}

We also have a left analogue of \Cref{constr:struct-lift-right-comp}.
\begin{construction}\label{constr:struct-lift-left-comp}
  Fix a composable pair of maps $U' \xrightarrow{i'} U \xrightarrow{i} V$ and a
  map $E \to B$.
  Then, we define a map
  \begin{equation*}
    \left(
      \begin{tikzcd}[cramped, row sep=small, column sep=small]
        U' \ar[d, "{i'}"'] \\ U
      \end{tikzcd}
      \squareslash
      \begin{tikzcd}[cramped, row sep=small, column sep=small]
        E \ar[d] \\ B
      \end{tikzcd}
    \right)
    \times
    \left(
      \begin{tikzcd}[cramped, row sep=small, column sep=small]
        U \ar[d, "i"'] \\ V
      \end{tikzcd}
      \squareslash
      \begin{tikzcd}[cramped, row sep=small, column sep=small]
        E \ar[d, ""] \\ B
      \end{tikzcd}
    \right)
    \xrightarrow{\qquad}
    \left(
      \begin{tikzcd}[cramped, row sep=small, column sep=small]
        U' \ar[d, "{i'}"'] \\ U \ar[d, "{i}"'] \\ V
      \end{tikzcd}
      \squareslash
      \begin{tikzcd}[cramped, row sep=small, column sep=small]
        E \ar[d] \\ B
      \end{tikzcd}
    \right)
  \end{equation*}
  that sends a pair $F' \in (U' \to U) \squareslash (E \to B)$ and
  $F \in (U \to V) \squareslash (E \to B)$ to the lifting structure which solves
  lifting problems $(u',v)$ with the consecutive lift
  $F(F'(u', v \cdot (X \times i)), v)$.
  \begin{equation*}
    \begin{tikzcd}[cramped]
      {X \times U'} &&& E \\
      {X \times U} \\
      {X \times V} &&& B
      \arrow["{u'}", color=red0, dashed, from=1-1, to=1-4]
      \arrow["{X \times i'}"', color=darkblue0, from=1-1, to=2-1]
      \arrow["p", color=darkblue0, from=1-4, to=3-4]
      \arrow["{F'(u', v \cdot  (X \times i))}"{description}, color=yellow0, dotted, from=2-1, to=1-4]
      \arrow["{X \times i}"', color=darkblue0, from=2-1, to=3-1]
      \arrow["{F(F'(u', v \cdot  (X \times i)), v)}"{description, pos=0.3}, color=yellow0, dotted, from=3-1, to=1-4]
      \arrow["v"', color=red0, dashed, from=3-1, to=3-4]
    \end{tikzcd}
  \end{equation*}
\end{construction}

We next show the structured version of the fact that the left maps to lifting
problems are closed under retracts.
\begin{construction}\label{constr:struct-lift-left-retract}
  Suppose one has a retract of composable pairs of maps.
  \begin{equation*}
    \begin{tikzcd}[cramped]
      {U_0} & U & {U_0} \\
      {V_0} & V & {V_0} \\
      & {V_0'} & {V'} & {V_0'}
      \arrow["m", from=1-1, to=1-2]
      \arrow[from=1-1, to=2-1]
      \arrow["n", from=1-2, to=1-3]
      \arrow[from=1-2, to=2-2]
      \arrow[from=1-3, to=2-3]
      \arrow["s"{description}, from=2-1, to=2-2]
      \arrow[from=2-1, to=3-2]
      \arrow["r"{description}, from=2-2, to=2-3]
      \arrow[from=2-2, to=3-3]
      \arrow[from=2-3, to=3-4]
      \arrow["{s'}"', from=3-2, to=3-3]
      \arrow["{r'}"', from=3-3, to=3-4]
    \end{tikzcd}
  \end{equation*}
  Then, we define a map
  \begin{equation*}
    \left(
      \begin{tikzcd}[cramped, row sep=small, column sep=small]
        U \ar[d] \\ V \ar[r] & V'
      \end{tikzcd}
      \squareslash
      \begin{tikzcd}[cramped, row sep=small, column sep=small]
        E \ar[d] \\ B
      \end{tikzcd}
    \right)
    \xrightarrow{\qquad}
    \left(
      \begin{tikzcd}[cramped, row sep=small, column sep=small]
        U_0 \ar[d] \\ V_0 \ar[r] & V_0'
      \end{tikzcd}
      \squareslash
      \begin{tikzcd}[cramped, row sep=small, column sep=small]
        E \ar[d] \\ B
      \end{tikzcd}
    \right)
  \end{equation*}
  sending each lifting structure $F$ from the left to the lifting structure that
  solves each lifting problem $(u_0,v_0)$ as below with the solution
  $F(u_0 \cdot (X \times n), v_0 \cdot (X \times r')) \cdot (X \times s)$.
  \begin{equation*}
    \begin{tikzcd}[cramped]
      {X \times U_0} & {X \times U} & {X \times U_0} \\
      \\
      {X \times V_0} & {X \times V} & {X \times V_0} && B \\
      & {X \times V_0'} & {X \times V'} & {X \times V_0'} \\
      &&&& {\underline{B}}
      \arrow["{X \times m}", from=1-1, to=1-2]
      \arrow[from=1-1, to=3-1]
      \arrow["{X \times n}", from=1-2, to=1-3]
      \arrow[from=1-2, to=3-2]
      \arrow[from=1-3, to=3-3]
      \arrow["{X \times s}"{description}, from=3-1, to=3-2]
      \arrow[from=3-1, to=4-2]
      \arrow["{X \times r}"{description}, from=3-2, to=3-3]
      \arrow[from=3-2, to=4-3]
      \arrow[from=3-3, to=4-4]
      \arrow[from=3-5, to=5-5]
      \arrow["{X \times s'}"', from=4-2, to=4-3]
      \arrow["{X \times r'}"', from=4-3, to=4-4]
      \arrow["{v_0}"', dashed, from=4-4, to=5-5, color=red0]
      \arrow[crossing over, curve={height=-36pt}, dotted, from=3-1, to=3-5, color=yellow0]
      \arrow[crossing over,
      "{F(u_0 \cdot (X \times n), v_0 \cdot (X \times r'))}"{description, pos=0.6},
      curve={height=16pt}, dotted, from=3-2, to=3-5, color=yellow0]
      \arrow["u_0", dashed, from=1-3, to=3-5, color=red0]
    \end{tikzcd}
  \end{equation*}
\end{construction}

By chaining \Cref{constr:struct-lift-right-res,constr:struct-lift-right-pb}, we
can observe that if $U \to V$ lifts uniformly on the left against $E \to B$ then
$U \to V$ also lifts uniformly on the left against any pullback $E' \to B'$ of
$E \to B$.
Dually, one may wonder if one can induce uniform lifting structures on pullbacks
of the left map $U \to V$.

First, recall that an exponentiable map $p \colon C \to D$ is a map where the
pullback functor $p^* \colon \sfrac{\bC}{D} \to \sfrac{\bC}{C}$ admits both the
left and right adjoint, respectively given by postcomposition and pushforwards.
\begin{equation*}
  \begin{tikzcd}[cramped]
    {\sfrac{\bC}{D}} & {\sfrac{\bC}{C}}
    \arrow[""{name=0, anchor=center, inner sep=0}, "{p^*}"{description}, from=1-1, to=1-2]
    \arrow[""{name=1, anchor=center, inner sep=0}, "{p_!}"', curve={height=18pt}, from=1-2, to=1-1]
    \arrow[""{name=2, anchor=center, inner sep=0}, "{p_*}"{description}, curve={height=-18pt}, from=1-2, to=1-1]
    \arrow["\dashv"{anchor=center, rotate=-91}, draw=none, from=0, to=2]
    \arrow["\dashv"{anchor=center, rotate=-89}, draw=none, from=1, to=0]
  \end{tikzcd}
\end{equation*}
We also recall the result from model category that if fibrations are stable
under pushforwards along fibrations then trivial cofibrations are stable under
pullback along fibrations.
By reproducing this result in a structured setting, we provide a construction
that induces uniform lifting structures on pullbacks of the left map.
\begin{construction}\label{constr:struct-tc-pullback-stable}
  Fix a generic left map $i \colon U \to V$ and a generic right map $E \to B$
  along with an exponentiable map $t \colon W \to V$ and a pullback
  $t^*i \colon t^*U \to W$ of $U \to V$ along $t$.
  \begin{equation*}
    \begin{tikzcd}[cramped]
      {t^*U} & U \\
      W & V
      \arrow[from=1-1, to=1-2]
      \arrow["{t^*i}"', from=1-1, to=2-1]
      \arrow["\lrcorner"{anchor=center, pos=0.15, scale=1.5}, draw=none, from=1-1, to=2-2]
      \arrow["i", from=1-2, to=2-2]
      \arrow["t"', from=2-1, to=2-2]
    \end{tikzcd}
  \end{equation*}

  We show that the pullback $t^*U \to W$ lifts uniformly against a generic right map
  $E \to B$ when the generic left map $U \to V$ lifts uniformly against the map
  $t_*(W \times E) \to t_*(W \times B)$ obtained from generic right map
  $E \to B \in \sfrac{\bC}{1}$ by applying the polynomial functor of
  $1 \leftarrow W \to V \to 1$.
  That is, we construct a map
  \begin{equation*}
    \left(
      \begin{tikzcd}[cramped, row sep=small] U \ar[d] \\ V \end{tikzcd}
      \squareslash
      \begin{tikzcd}[cramped, row sep=small] t_*(W \times E) \ar[d] \\ t_*(W \times B) \end{tikzcd}
    \right)
    \xrightarrow{\qquad}
    \left(
      \begin{tikzcd}[cramped, row sep=small] t^*U \ar[d] \\ V' \end{tikzcd}
      \squareslash
      \begin{tikzcd}[cramped, row sep=small] E \ar[d] \\ B \end{tikzcd}
    \right)
  \end{equation*}

  Let $F$ be a lifting structure of $U \to V$ against
  $t_*(W \times E) \to t_*(W \times B)$.
  Now assume that we are given a lifting problem $\textcolor{red0}{(u,v)}$ of
  $X \times t^*U \to X \times W$ against $E \to B$ as in the curved back face.
  \begin{equation*}
    \begin{tikzcd}[cramped]
      {X \times t^*U} &&&& {W \times E} && E \\
      & {X \times U} &&&& {t_*(W \times E)} \\
      {X \times W} &&&& {W \times B} && B \\
      &&&& W \\
      & {X \times V} &&&& {t_*(W \times B)} \\
      &&&&& V
      \arrow["{(t^*i \cdot \proj_2, u)}"{description, pos=0.7}, color=green0, dashed, from=1-1, to=1-5]
      \arrow["u"{description}, color=red0, curve={height=-12pt}, dashed, from=1-1, to=1-7]
      \arrow[from=1-1, to=2-2]
      \arrow[from=1-1, to=3-1]
      \arrow[from=1-5, to=1-7]
      \arrow[from=1-5, to=3-5]
      \arrow[from=1-7, to=3-7]
      \arrow[from=2-2, to=5-2]
      \arrow[from=2-6, to=5-6]
      \arrow["{F((t^*i \cdot \proj_2, u)^\dagger, (i \cdot \proj_2, u)^\dagger)^\dagger}"{description, pos=0.7}, color=yellow0, dotted, from=3-1, to=1-5]
      \arrow["{(\proj_2, v)}"{description, pos=0.7}, color=green0, dashed, from=3-1, to=3-5]
      \arrow["v"{description}, color=red0, curve={height=-12pt}, dashed, from=3-1, to=3-7]
      \arrow[from=3-1, to=4-5]
      \arrow[from=3-1, to=5-2]
      \arrow[from=3-5, to=3-7]
      \arrow[from=3-5, to=4-5]
      \arrow["t"{description, pos=0.3}, from=4-5, to=6-6]
      \arrow[from=5-2, to=6-6]
      \arrow[from=5-6, to=6-6]
      \arrow["{F((t^*i \cdot \proj_2, u)^\dagger, (i \cdot \proj_2, u)^\dagger)}"{description, pos=0.25}, color=darkblue0, dotted, from=5-2, to=2-6]
      \arrow[crossing over, "{(t^*i \cdot \proj_2, u)^\dagger}"{description}, color=magenta0, dashed, from=2-2, to=2-6]
      \arrow[crossing over, "{(\proj_2, v)^\dagger}"{description}, color=magenta0, dashed, from=5-2, to=5-6]
    \end{tikzcd}
  \end{equation*}
  Then, we produce the required lifting solution by the following procedure:
  \begin{enumerate}
    \item Induce a lifting problem given by
    $\textcolor{green0}{((t^*i \cdot \proj_2, u), (\proj_2, v))}$ of
    $X \times t^*U \to X \times W$ against $W \times E \to W \times B$.
    \item Because $X \times -$ preserves pullbacks, the pullback of
    $X \times U \to X \times V \in \sfrac{\bC}{V}$ under $t \colon W \to V$ is
    $X \times t^*U \to X \times W \in \sfrac{\bC}{W}$.
    So we can transpose
    $\textcolor{green0}{((t^*i \cdot \proj_2, u), (\proj_2, v))}$ to get a lifting
    problem
    $\textcolor{magenta0}{((t^*i \cdot \proj_2, u)^\dagger, (\proj_2, v)^\dagger)}$
    to get a solution.
    \item By the original solution $F$, we obtain a solution
    $\textcolor{darkblue0}{F((t^*i \cdot \proj_2, u)^\dagger, (\proj_2, v)^\dagger)}$.
    \item Transposing it back we obtain the required solution
    $\textcolor{yellow0}{F((t^*i \cdot \proj_2, u)^\dagger, (\proj_2,
      v)^\dagger)^\dagger}$.
  \end{enumerate}
  Because the transpose operation is natural, the uniformity structure of $F$
  shows that this indeed induces a structured lift.
\end{construction}

As a consequence, we obtain the following by combining various constructions
above.
\begin{theorem}\label{thm:struct-tc-pullback-stable}
  Fix a pullback $t^*U \to W \in \sfrac{\bC}{C}$ of a map
  $U \to V \in \sfrac{\bC}{C}$ along a map
  $t \colon W \to V \in \sfrac{\bC}{C}$.

  \begin{center}
    \begin{minipage}{0.45\linewidth}
      \begin{equation*}
        \begin{tikzcd}[cramped]
          {t^*U} & U \\
          W & V
          \arrow[from=1-1, to=1-2]
          \arrow[""', from=1-1, to=2-1]
          \arrow["\lrcorner"{anchor=center, pos=0.15, scale=1.5}, draw=none, from=1-1, to=2-2]
          \arrow["", from=1-2, to=2-2]
          \arrow["t"', from=2-1, to=2-2]
        \end{tikzcd}
      \end{equation*}
    \end{minipage}
    \begin{minipage}{0.45\linewidth}
      \begin{equation*}
        \begin{tikzcd}[cramped]
          t_*(W \times_C E) \ar[d] \ar[r, ""] \ar[rd, draw=none, "{\lrcorner}"{pos=0.15,scale=1.5}]
          & E \ar[d]
          \\ t_*(W \times_C B) \ar[r, ""]
          & B
        \end{tikzcd}
      \end{equation*}
    \end{minipage}
  \end{center}

  If $U \to V \in \bC$ lifts uniformly against a map $E \to B \in \bC$ and the
  pushed forwards map $t_*(W \times E) \to t_*(W \times B)$ occurs as a pullback
  of $E \to B$ then $t^*U \to W \in \bC$ lifts uniformly against
  $E \to B \in \bC$ as well.
\end{theorem}
\begin{proof}
  By
  \Cref{constr:struct-tc-pullback-stable,constr:struct-lift-right-res,constr:struct-lift-right-pb}.
\end{proof}


\section{Rebasing Lifting Structures}\label{subsec:struct-lift-rebase}
We now work with lifting structures in slice categories and show that they are
stable under the pullback and post-composition change of base operations.

We begin with spelling out what it means to be a lifting structure in a slice
category.
\begin{definition}\label{def:local-struct-lift}
  Fix an object $C \in \bC$.
  Given maps $U \to V$ and $V \to V'$ and $E \to B$ in the slice
  $\sfrac{\bC}{C}$, the set of structured lifts in $\sfrac{\bC}{C}$ of $U \to V$
  restricted on the left along $V \to V'$ against $E \to B$ is denoted as on the
  left below.
  \begin{center}
    \begin{minipage}{0.45\linewidth}
      \begin{equation*}
        {\begin{tikzcd}[cramped, column sep=small, row sep=small]
            U \ar[d] \\ V \ar[r] & V'
          \end{tikzcd}}
        \fracsquareslash{C}
        {\begin{tikzcd}[cramped, column sep=small, row sep=small]
            E \ar[d] \\ B
          \end{tikzcd}}
      \end{equation*}
    \end{minipage}
    \begin{minipage}{0.45\linewidth}
      \begin{equation*}
        {\begin{tikzcd}[cramped, row sep=small, column sep=small]
            U \ar[d] \\ V
          \end{tikzcd}}
        \fracsquareslash{C}
        {\begin{tikzcd}[cramped, row sep=small, column sep=small]
            & E \ar[d] \\ B' \ar[r] & B
          \end{tikzcd}}
      \end{equation*}
    \end{minipage}
  \end{center}
  Dually, given $U \to V$ and $B' \to B$ and $E \to B$ in the slice
  $\sfrac{\bC}{C}$, the set of structured lifts in $\sfrac{\bC}{C}$ of $U \to V$
  restricted on the right along $B' \to B$ against $E \to B$ is denoted as on the
  right above.
\end{definition}

Because lifting structures are inherently a limiting, or a ``right-sided
concept'', it is unsurprising they are stable under pullback.

\begin{construction}\label{constr:struct-lift-pb}
  Fix a map $\phi \colon D \to C$ along with maps
  $U \to V \to V' \in \sfrac{\bC}{C}$ and $E \to B \in \sfrac{\bC}{C}$ where
  $U,V,V'$ are exponential objects in the slice over $C$ so that by the
  Beck-Chevalley condition, $\phi^*U,\phi^*V,\phi^*V'$ are exponential objects
  in the slice over $D$.
  We define a map
  \begin{equation*}
    \left({\begin{tikzcd}[cramped, column sep=small, row sep=small]
          U \ar[d] \\ V \ar[r] & V'
        \end{tikzcd}}
      \fracsquareslash{C}
      {\begin{tikzcd}[cramped, column sep=small, row sep=small]
          E \ar[d] \\ B
        \end{tikzcd}}\right)
    \xrightarrow{\qquad}
    \left({\begin{tikzcd}[cramped, column sep=small, row sep=small]
          \phi^*U \ar[d] \\ \phi^*V \ar[r] & \phi^*V'
        \end{tikzcd}}
      \fracsquareslash{D}
      {\begin{tikzcd}[cramped, column sep=small, row sep=small]
          \phi^*E \ar[d] \\ \phi^*B
        \end{tikzcd}}\right)
  \end{equation*}
  by nothing that the pullback functor
  $\phi^* \colon \sfrac{\bC}{C} \to \sfrac{\bC}{D}$ preserves limits and
  exponentials so, one has the pullbacks
  \begin{equation*}
    \begin{tikzcd}[cramped, row sep=small, column sep=small]
      {\textcolor{red0}{[\phi^*U,\phi^*E]_D \times_{[\phi^*U,\phi^*B]_D} [\phi^*V',\phi^*B]_D}}
      && {\textcolor{red0}{[U,E]_C \times_{[U,B]_C} [V',B]_C}}
      \\
      & {\textcolor{yellow0}{[\phi^*V,\phi^*E]_D}} && {\textcolor{yellow0}{[V,E]_C}} \\
      {\textcolor{darkblue0}{[\phi^*U,\phi^*E]_D \times_{[\phi^*U,\phi^*B]_D} [\phi^*V,\phi^*B]_D}}
      && {\textcolor{darkblue0}{[U,E]_C \times_{[U,B]_C} [V,B]_C}}
      \\
      D && C
      \arrow["\lrcorner"{anchor=center, pos=0.05, scale=1.5, rotate=0}, draw=none, from=1-1, to=3-3]
      \arrow["\lrcorner"{anchor=center, pos=0.05, scale=1.5, rotate=0}, draw=none, from=2-2, to=3-3]
      \arrow["\lrcorner"{anchor=center, pos=0.05, scale=1.5, rotate=0}, draw=none, from=3-1, to=4-3]
      \arrow[from=3-1, to=4-1]
      \arrow[from=3-3, to=4-3]
      \arrow["\phi"', from=4-1, to=4-3]
      \arrow[color={red0}, from=1-3, to=3-3]
      \arrow[color={red0}, from=1-1, to=3-1]
      \arrow[color={yellow0}, from=2-2, to=3-1]
      \arrow[color={yellow0}, from=2-4, to=3-3]
      \arrow[from=3-1, to=3-3]
      \arrow[from=1-1, to=1-3]
      \arrow[from=2-2, to=2-4, crossing over]
    \end{tikzcd}
  \end{equation*}
  The required map is then given by \Cref{lem:stable-left-rep}.

  Similarly, for maps $U \to V$ and $B' \to B \leftarrow E$ where $U,V$ are
  exponentiable objects in $\sfrac{\bC}{C}$, we can construct a map
  \begin{equation*}
    \left(\begin{tikzcd}[cramped, row sep=small, column sep=small]
        U \ar[d] \\ V
      \end{tikzcd}
      \fracsquareslash{C}
      \begin{tikzcd}[cramped, row sep=small, column sep=small]
        & E \ar[d] \\ B' \ar[r] & B
      \end{tikzcd}\right)
    \xrightarrow{\qquad}
    \left(\begin{tikzcd}[cramped, row sep=small, column sep=small]
        \phi^*U \ar[d] \\ \phi^*V
      \end{tikzcd}
      \fracsquareslash{D}
      \begin{tikzcd}[cramped, row sep=small, column sep=small]
        & \phi^*E \ar[d] \\ \phi^*B' \ar[r] &
        \phi^*B
      \end{tikzcd}\right)
  \end{equation*}
\end{construction}

We next proceed to show that the rebasing by post-composition also preserves
certain forms of lifting structures.
To do so, recall that the pushforward realises the external left adjoint to the
pullback functor as also the internal left adjoint, made precise in the
following sense.
\begin{lemma}\label{lem:pushforward-pb-left-adj}
  Fix an exponentiable map $p \colon C \to D$.
  Then,
  \begin{enumerate}
    \item \label{itm:pushforward-pb-left-adj-1}
    For any $X \to C \in \sfrac{\bC}{C}$ and $Z \to D \in \sfrac{\bC}{D}$,
    one has $p_!(X \times_C p^*Z) \cong p_!X \times_D Z$, where $p_!$ is the
    left adjoint to the pullback functor.
    \item \label{itm:pushforward-pb-left-adj-2} For any exponentiable
    $X \to C \in \sfrac{\bC}{C}$ and $Y \to D \in \sfrac{\bC}{D}$, one has
    $p_*[X, p^*Y]_C \cong [p_!X, Y]_D$, where $p_*$ is the pushforward
    (i.e. right adjoint the pullback) along $p$.
  \end{enumerate}
\end{lemma}
\begin{proof}
  \Cref{itm:pushforward-pb-left-adj-1} is directly by the pullback lemma:
  \begin{equation*}
    \begin{tikzcd}
      \bullet \\
      X & {p^*Z} & Z \\
      & C & D
      \arrow[from=1-1, to=2-1]
      \arrow[from=1-1, to=2-2]
      \arrow[from=2-1, to=3-2]
      \arrow[from=2-2, to=2-3]
      \arrow[from=2-2, to=3-2]
      \arrow["\lrcorner"{anchor=center, pos=0.05, scale=1.5}, draw=none, from=1-1, to=3-2]
      \arrow["\lrcorner"{anchor=center, pos=0.05, scale=1.5}, draw=none, from=2-2, to=3-3]
      \arrow[from=2-3, to=3-3]
      \arrow[from=3-2, to=3-3, "p"']
    \end{tikzcd}
  \end{equation*}

  \Cref{itm:pushforward-pb-left-adj-2} is now by
  \Cref{itm:pushforward-pb-left-adj-1} and representability, because for each
  $Z \to D \in \sfrac{\bC}{D}$, one has
  \begin{align*}
    \sfrac{\bC}{D}(Z, p_*[X, p^*Y]_C)
    &= \sfrac{\bC}{C}(p^*Z, [X, p^*Y]_C) \\
    &= \sfrac{\bC}{C}(p^*Z \times_C X, p^*Y) \\
    &= \sfrac{\bC}{D}(p_!(p^*Z \times_C X), Y) \\
    &= \sfrac{\bC}{D}(Z \times_D p_!X, Y) \\
    \sfrac{\bC}{D}(Z, p_*[X, p^*Y]_C)
    &= \sfrac{\bC}{D}(Z, [p_!X, Y]_D)
  \end{align*}
\end{proof}
\begin{construction}\label{constr:struct-lift-postcomp}
  Fix a map $p \colon C \to D$.
  Let there be maps $U \to V$ and $V \to V'$ in $\sfrac{\bC}{C}$ along with a
  map $E \to B \in \sfrac{\bC}{D}$, where $U,V,V'$ are exponentiable in the
  slice category over $C$.
  We construct a map
  \begin{equation*}
    \left({\begin{tikzcd}[cramped, column sep=small, row sep=small]
          U \ar[d] \\ V \ar[r] & V'
        \end{tikzcd}}
      \fracsquareslash{C}
      {\begin{tikzcd}[cramped, column sep=small, row sep=small]
          {p^*E} \ar[d] \\ {p^*B}
        \end{tikzcd}}\right)
    \xrightarrow{\qquad}
    \left({\begin{tikzcd}[cramped, column sep=small, row sep=small]
          p_!U \ar[d] \\ p_!V \ar[r] & p_!V'
        \end{tikzcd}}
      \fracsquareslash{D}
      {\begin{tikzcd}[cramped, column sep=small, row sep=small]
          E \ar[d] \\ B
        \end{tikzcd}}\right)
  \end{equation*}
  by \Cref{lem:pushforward-pb-left-adj}, which states that right adjoint
  $p_* \colon \sfrac{\bC}{C} \to \sfrac{\bC}{D}$ maps
  \begin{equation*}\small
    \begin{tikzcd}[row sep=small, column sep=small, cramped]
      {\textcolor{red0}{[U, p^*E]_C \times_{[U,p^*B]_C} [V', p^*B]_C}} \\
      & {\textcolor{yellow0}{[V,p^*E]_C}} \\
      {\textcolor{darkblue0}{[U, p^*E]_C \times_{[U, p^*B]_C} [V, p^*B]_C}}
      && {\textcolor{red0}{[p_!U, E]_D \times_{[p_!U, B]_D} [p_!V', B]_D}} \\
      C &&& {\textcolor{yellow0}{[p_!V,E]_D}} \\
      && {\textcolor{darkblue0}{[p_!U,E]_D \times_{[p_!U, B]_D} [p_!V, B]_D}} \\
      && D
      \arrow[from=3-1, to=4-1]
      \arrow["{p_*}", shorten <=39pt, shorten >=39pt, maps to, from=3-1, to=5-3]
      \arrow[draw={red0}, from=1-1, to=3-1]
      \arrow[color={red0}, from=3-3, to=5-3]
      \arrow[draw={yellow0}, from=2-2, to=3-1]
      \arrow[color={yellow0}, from=4-4, to=5-3]
      \arrow[from=5-3, to=6-3]
      \arrow["p"', from=4-1, to=6-3]
    \end{tikzcd}
  \end{equation*}

  For the same reason, if there are maps $U \to V$ in $\sfrac{\bC}{C}$ where
  $U,V$ are exponentiable in the slice over $C$ and maps $B' \to B$ and
  $E \to B$ in $\sfrac{\bC}{D}$ then one also has a map
  \begin{equation*}
    \left(\begin{tikzcd}[cramped, row sep=small, column sep=small]
        U \ar[d] \\ V
      \end{tikzcd}
      \fracsquareslash{C}
      \begin{tikzcd}[cramped, row sep=small, column sep=small]
        & p^*E \ar[d] \\ p^*B' \ar[r] & p^*B
      \end{tikzcd}\right)
    \xrightarrow{\qquad}
    \left(\begin{tikzcd}[cramped, row sep=small, column sep=small]
        p_!U \ar[d] \\ p_!V
      \end{tikzcd}
      \fracsquareslash{D}
      \begin{tikzcd}[cramped, row sep=small, column sep=small]
        & E \ar[d] \\ B' \ar[r] & B
      \end{tikzcd}\right)
  \end{equation*}
\end{construction}

We also show a local version of \Cref{constr:struct-tc-pullback-stable}.
Because \Cref{constr:struct-tc-pullback-stable} mentions about pushforwards, we
need to recall that pushforwards in slice categories are just computed as
pushforwards in the ambient category.
\begin{lemma}\label{lem:pshfw-slice}
  Fix $C \in \bC$ and a map $t \colon U \to V \in \sfrac{\bC}{C}$.
  Then the pushforwards along $t$ in $\sfrac{\bC}{C}$ is the pushforwards along
  $t$ in the ambient category $\bC$.
\end{lemma}
\begin{proof}
  This is because pullbacks in the slice category are exactly pullbacks in the
  ambient category.
\end{proof}

Therefore, we have a local version of \Cref{constr:struct-tc-pullback-stable}.
\begin{construction}\label{constr:struct-tc-pullback-stable-slice}
  Fix an object $C \in \bC$ and a generic left map
  $i \colon U \to V \in \sfrac{\bC}{C}$ and a generic right map
  $E \to B \in \sfrac{\bC}{C}$, where $U,V$ are exponentiable in the slice,
  along with an exponentiable map $t \colon W \to V$ and a pullback
  $t^*U \to W \in \sfrac{\bC}{C}$ of $U \to V$ along $t$.

  Then, we construct a map
  \begin{equation*}
    \left(
      \begin{tikzcd}[cramped, row sep=small] U \ar[d] \\ V \end{tikzcd}
      \fracsquareslash{C}
      \begin{tikzcd}[cramped, row sep=small] t_*(W \times_C E) \ar[d] \\ t_*(W \times_C B) \end{tikzcd}
    \right)
    \xrightarrow{\qquad}
    \left(
      \begin{tikzcd}[cramped, row sep=small] t^*U \ar[d] \\ V' \end{tikzcd}
      \fracsquareslash{C}
      \begin{tikzcd}[cramped, row sep=small] E \ar[d] \\ B \end{tikzcd}
    \right)
  \end{equation*}
  by applying \Cref{constr:struct-tc-pullback-stable} and using
  \Cref{lem:pshfw-slice} to note that pushforwards in the slice $\sfrac{\bC}{C}$
  is just the pushforwards in $\bC$.
\end{construction}

As a consequence, we obtain the following local version of
\Cref{thm:struct-tc-pullback-stable}.
\begin{theorem}
  Fix an object $C \in \bC$ and a pullback $t^*U \to W \in \sfrac{\bC}{C}$ of a
  map $U \to V \in \sfrac{\bC}{C}$ between exponentiable objects along an
  exponentiable map $t \colon W \to V \in \sfrac{\bC}{C}$.

  \begin{center}
    \begin{minipage}{0.45\linewidth}
      \begin{equation*}
        \begin{tikzcd}[cramped]
          {t^*U} & U \\
          W & V
          \arrow[from=1-1, to=1-2]
          \arrow[""', from=1-1, to=2-1]
          \arrow["\lrcorner"{anchor=center, pos=0.15, scale=1.5}, draw=none, from=1-1, to=2-2]
          \arrow["", from=1-2, to=2-2]
          \arrow["t"', from=2-1, to=2-2]
        \end{tikzcd}
      \end{equation*}
    \end{minipage}
    \begin{minipage}{0.45\linewidth}
      \begin{equation*}
        \begin{tikzcd}[cramped]
          t_*(W \times_C E) \ar[d] \ar[r, ""] \ar[rd, draw=none, "{\lrcorner}"{pos=0.15,scale=1.5}]
          & E \ar[d]
          \\ t_*(W \times_C B) \ar[r, ""]
          & B
        \end{tikzcd}
      \end{equation*}
    \end{minipage}
  \end{center}

  If $U \to V \in \sfrac{\bC}{C}$ lifts uniformly against a map
  $E \to B \in \sfrac{\bC}{C}$ and the pushed forwards map
  $t_*(W \times_C E) \to t_*(W \times_C B)$ occurs as a pullback of $E \to B$
  then $t^*U \to W \in \sfrac{\bC}{C}$ lifts uniformly against
  $E \to B \in \sfrac{\bC}{C}$ as well.
\end{theorem}
\begin{proof}
  By
  \Cref{constr:struct-tc-pullback-stable-slice,constr:struct-lift-right-res,constr:struct-lift-right-pb}.
\end{proof}

\subsection*{Slice Lifting Structures in CwRs}
In order to categorically axiomatise pattern matching operations such as
$\MsJ$-elimination in the framework CwRs, we require the construction of CwRs
freely extended with a formal lifting structure in slice categories.
The idea is the same as \Cref{constr:free-lift-ext}.


\begin{construction}\label{constr:free-lift-ext-slice}
  Let $\bC$ be a CwR and fix maps $U \to V \to V'$ between exponential objects
  and $E \to B$ the slice over a fixed object $C \in \bC$.
  Define the CwR obtained from $\bC$ by freely extending with a lifting
  structure of $U \to V$ against $E \to B$ restricted along $V \to V'$ in the
  slice over $C$ as the following bipushout in $\CwR$ using
  \Cref{constr:cwr-free-quotient}.
  \begin{lrbox}{\UVVp}
    \scriptsize\begin{tikzcd}[column sep=small, row sep=small]
      U \ar[d] \\ V \ar[r] & V'
    \end{tikzcd}
  \end{lrbox}
  \begin{lrbox}{\EB}
    \scriptsize\begin{tikzcd}[column sep=small, row sep=small]
      E \ar[d] \\ B
    \end{tikzcd}
  \end{lrbox}
  \begin{lrbox}{\UV}
    \scriptsize\begin{tikzcd}[column sep=small, row sep=small]
      U \ar[d] \\ V
    \end{tikzcd}
  \end{lrbox}
  \begin{lrbox}{\EBBp}
    \scriptsize\begin{tikzcd}[column sep=small, row sep=small]
      & E \ar[d] \\ B' \ar[r] & B
    \end{tikzcd}
  \end{lrbox}
  \begin{equation*}
    \bC\left[
      {\scriptsize\begin{tikzcd}[column sep=small, row sep=small]
          U \ar[d] \\ V \ar[r] & V'
        \end{tikzcd}}
      \fracsquareslash{C}
      {\scriptsize\begin{tikzcd}[column sep=small, row sep=small]
          E \ar[d] \\ B
        \end{tikzcd}}
    \right]
    \coloneqq \bC
    \cup_{\set{\textcolor{red0}{\bullet \to {}} \textcolor{darkblue0}{\bullet} \textcolor{yellow0}{{} \leftarrow \bullet}}}
    \left\{
      {\scriptsize\begin{tikzcd}[column sep=small, row sep=small]
          \textcolor{red0}{[U,E]_C \times_{[U,B]_C} [V',B]_C} \ar[d, color=red0]
          \ar[r]
          &
          \textcolor{yellow0}{[V,E]_C} \ar[ld, color=yellow0]
          \\
          \textcolor{darkblue0}{[U,E]_C \times_{[U,B]_C} [V,B]_C}
        \end{tikzcd}}
    \right\}
  \end{equation*}
  %
  %
  %
  Dually, if $U \to V$ is a map between exponentiable objects in
  $\sfrac{\bC}{C}$ and one has maps $E \to B$ along with $B' \to B$ in
  $\sfrac{\bC}{C}$ then the CwR with a formal lifting structure of $U \to V$
  against $E \to B$ restricted along $B' \to B$ in the slice over $C$ is
  obtained by the following bipushout in $\CwR$ using
  \Cref{constr:cwr-free-quotient}.
  \begin{equation*}
    \bC\left[
      {\scriptsize\begin{tikzcd}[column sep=small, row sep=small]
          U \ar[d] \\ V
        \end{tikzcd}}
      \fracsquareslash{C}
      {\scriptsize\begin{tikzcd}[column sep=small, row sep=small]
          & E \ar[d] \\ B' \ar[r] & B
        \end{tikzcd}}
    \right]
    \coloneqq \bC
    \cup_{\set{\textcolor{red0}{\bullet \to {}} \textcolor{darkblue0}{\bullet} \textcolor{yellow0}{{} \leftarrow \bullet}}}
    \left\{
      {\scriptsize\begin{tikzcd}[column sep=small, row sep=small]
          \textcolor{red0}{[U,E]_C \times_{[U,B]_C} [V,B']_C} \ar[d, color=red0]
          \ar[r]
          &
          \textcolor{yellow0}{[V,E]_C} \ar[ld, color=yellow0]
          \\
          \textcolor{darkblue0}{[U,E]_C \times_{[U,B]_C} [V,B]_C}
        \end{tikzcd}}
    \right\}
  \end{equation*}
\end{construction}

\begin{theorem}\label{thm:free-lift-ext-slice}
  Let $\bC$ be a CwR and fix maps $U \to V \to V'$ between representable objects
  (i.e. objects whose maps into the terminal object are representable) and
  $E \to B$ in $\bC$.
  Then, isomorphism classes of maps as on the left below correspond bijectively
  to isomorphism classes of maps $F \colon \bC \to \bD \in \CwR$ equipped with a
  choice of a lifting structure of $FU \to FV$ against $FE \to FB$ restricted
  along $FV \to FV'$ over the slice $FC$.
  \begin{center}
    \begin{minipage}{0.45\linewidth}
      \begin{equation*}
        \bC\left[
          {\scriptsize\begin{tikzcd}[column sep=small, row sep=small]
              U \ar[d] \\ V \ar[r] & V'
            \end{tikzcd}}
          \fracsquareslash{C}
          {\scriptsize\begin{tikzcd}[column sep=small, row sep=small]
              E \ar[d] \\ B
            \end{tikzcd}}
        \right] \to \bD \in \CwR
      \end{equation*}
    \end{minipage}
    \begin{minipage}{0.45\linewidth}
      \begin{equation*}
        \bC\left[
          {\scriptsize\begin{tikzcd}[column sep=small, row sep=small]
              U \ar[d] \\ V \ar[r] & V'
            \end{tikzcd}}
          \fracsquareslash{C}
          {\scriptsize\begin{tikzcd}[column sep=small, row sep=small]
              E \ar[d] \\ B
            \end{tikzcd}}
        \right] \to \bD \in \CwR
      \end{equation*}
    \end{minipage}
  \end{center}
  Dually, given $B' \to B$ in $\bC$, isomorphism classes of maps as on the right
  above correspond to isomorphism classes of maps
  $F \colon \bC \to \bD \in \CwR$ equipped with a choice of a lifting structure
  of $FU \to FV$ against $FE \to FB$ restricted along $FB' \to FB$ over the
  slice $FC$.
\end{theorem}
\begin{proof}
  Identical to \Cref{thm:free-lift-ext}.
\end{proof}


\section{Structured Approximations of Pushout-Product}\label{sec:pushout-product-approx}
In preparation for the structured version of Leibniz transposes in the
subsequent \Cref{subsec:struct-ltrans}, we first motivate and introduce a structured
approximation of the pushout-product.

As motivation, suppose we are given two maps, which we think of as ``boundary
inclusions'' and denote their domains and codomains using the boundary symbol,
$\partial V \to V$ and $\partial L \to L$.
Usually, the boundary $\partial(V \times L)$ of the product $V \times L$ is
constructed by the pushout-product
$\partial V \times L \cup V \times \partial L \to V \times L$.
However, because the pushout-product relies on the left concept of pushouts, and
we only assume our universe of discourse $\bC$ to be locally cartesian closed,
these pushouts may not exist in the first place.
Therefore, we would like to have a concept of maps that behaves sufficiently
like such a pushout-product.

The usual product boundary $\partial V \times L \cup V \times \partial L$, being
a pushout, has the universal property that maps
$\partial V \times L \cup V \times \partial L \to E$ are in bijective
correspondence with pairs of maps $\partial V \times L \to E$ and
$V \times \partial L \to E$.
But because in a cartesian closed category, products are left adjoint to
exponential, taking products preserve pushouts.
This means that for each $X \in \bC$, one in fact has
$X \times (\partial V \times L \cup V \times \partial L \to E) \cong X \times
\partial V \times L \cup X \times V \times \partial L$.

Therefore, we would like to say that an object $E$ believes an object
$\partial(V \times L)$ behaves like the product boundary when, for each
$X \in \bC$, one has a natural bijective correspondence between maps $f$ and
undercones
$(f|_{X \times \partial V \times L}, f|_{X \times V \times \partial L})$ as
follows.
\begin{center}
  \begin{minipage}{0.45\linewidth}
    \begin{equation*}
      X \times \partial(V \times L) \xrightarrow{f} E
    \end{equation*}
  \end{minipage}
  \begin{minipage}{0.45\linewidth}
    \begin{equation*}
      \begin{tikzcd}[cramped]
        {X \times \partial V \times \partial L} & {X \times V \times \partial L} \\
        {X \times \partial V \times L} & E
        \arrow[from=1-1, to=1-2]
        \arrow[from=1-1, to=2-1]
        \arrow["{f|_{X \times V \times \partial L}}", dashed, from=1-2, to=2-2]
        \arrow["{f|_{X \times \partial V \times L}}"', dashed, from=2-1, to=2-2]
      \end{tikzcd}
    \end{equation*}
  \end{minipage}
\end{center}
Such parameterised beliefs can be packaged into the following definition using
cartesian closeness.
\begin{definition}
  Given maps $\partial V \to V$ and $\partial L \to L$ between exponentiable
  objects in $\bC$, a \emph{product boundary approximation structure} on an
  exponentiable object $\partial(V \times L)$ relative to an object $E$ is an
  isomorphism
  \begin{equation*}
    \textcolor{darkblue0}{
      [\partial V \times L, E] \times_{[\partial V \times \partial L, E]} [V \times \partial L, E]
    }
    \xrightarrow{\cong}
    \textcolor{darkblue0}{[\partial(V \times L), E]}
  \end{equation*}
\end{definition}

Next, suppose $E$ and $B$ both believe $\partial(V \times L)$ behaves like the
product boundary.
Then, $E$ thinks that each map $f \colon X \times \partial(V \times L) \to E$ is
completely determined by the two components
$f|_{X \times \partial V \times L}$ and
$f|_{X \times V \times \partial L}$.
Now, given a map $p \colon E \to B$, we can post-compose to obtain a map
$pf \colon X \times \partial(V \times L) \to E \to B$.
Then, $B$ also thinks $pf$ is completely determined by the two components
$(pf)|_{X \times \partial V \times L}$ and
$(pf)|_{X \times V \times \partial L}$.
We want to say that composing with $p$ does not destroy the hallucinations $E$
and $B$ have about $\partial(V \times L) \to V \times L$ when the hallucinated
component restrictions $(pf)|_{X \times \partial V \times L}$ and
$(pf)|_{X \times V \times \partial L}$ of $B$ are respectively obtained by
composing $p$ with the hallucinated component restrictions
$f|_{X \times \partial V \times L}$ and $f|_{X \times V \times \partial L}$ of
$E$.
\begin{equation*}
  \begin{tikzcd}
    {X \times \partial V \times \partial L} & {X \times V \times \partial L} \\
    {X \times \partial V \times L} & {X \times \partial(V \times L)} \\
    && E \\
    &&& {B}
    \arrow[from=1-1, to=1-2]
    \arrow[from=1-1, to=2-1]
    \arrow[dotted, from=1-2, to=2-2]
    \arrow["{f|_{X \times V \times \partial L}}"{description, pos=0.6}, curve={height=-12pt}, dashed, from=1-2, to=3-3]
    \arrow["{(pf)|_{X \times V \times \partial L} = p(f|_{X \times V \times \partial L})}", curve={height=-30pt}, dashed, from=1-2, to=4-4]
    \arrow[dotted, from=2-1, to=2-2]
    \arrow["{f|_{X \times \partial V \times L}}"{description, pos=0.6}, curve={height=12pt}, dashed, from=2-1, to=3-3]
    \arrow["{(pf)|_{X \times \partial V \times L} = p(f|_{X \times \partial V \times L})}"', curve={height=30pt}, dashed, from=2-1, to=4-4]
    \arrow["f"{description}, dashed, from=2-2, to=3-3]
    \arrow["p", from=3-3, to=4-4]
  \end{tikzcd}
\end{equation*}
We once again can package such hallucinogenic maps using internal-Homs.
\begin{definition}
  Fix maps $\partial V \to V$ and $\partial L \to L$ between exponentiable
  objects in $\bC$.
  Suppose that an exponentiable object $\partial(V \times L)$ is equipped with
  product approximation structures $s_E$ and $s_B$ respectively relative to $E$
  and $B$.
  A map $p$ \emph{preserves boundary approximation structures} when the
  following diagram commutes.
  \begin{equation*}
    \begin{tikzcd}[cramped]
      {\textcolor{darkblue0}{[\partial V \times L, E] \times_{[\partial V \times \partial L, E]} [V \times \partial L, E]}} & {\textcolor{darkblue0}{[\partial(V \times L), E]}} \\
      {\textcolor{darkblue0}{[\partial V \times L, B] \times_{[\partial V \times \partial L, B]} [V \times \partial L, B]}} & {\textcolor{darkblue0}{[\partial(V \times L), B]}}
      \arrow["\cong", tail reversed, from=1-1, to=1-2]
      \arrow[color=darkblue0, from=1-1, to=2-1]
      \arrow[color=darkblue0, from=1-2, to=2-2]
      \arrow["\cong"', tail reversed, from=2-1, to=2-2]
    \end{tikzcd}
  \end{equation*}
\end{definition}

We have now defined what is means for an \emph{object} $\partial(V \times L)$ to
behave like the \emph{boundary object} of $V \times L$.
In the usual pushout-product definition, such boundary objects
$\partial V \times L \cup V \times \partial L$ are equipped with inclusions into
the actual product object $V \times L$.
We can also say that a map $\partial(V \times L) \to V \times L$ behaves like a
pushout-product inclusion, but again only relative to some object $B$ in a state
of delusion.
When we have such a map $\partial(V \times L) \to V \times L$ like this, given
any map $g \colon X \times V \times L \to B$, we can restrict it to get a map
$g|_{X \times \partial(V \times L)} \colon X \times \partial(V \times L) \to B$.
By hallucinations, $g|_{X \times \partial(V \times L)}$ is completely determined
by its further restrictions
$(g|_{X \times \partial(V \times L)})|_{X \times \partial V \times L}$ and
$(g|_{X \times \partial(V \times L)})|_{X \times V \times \partial L}$.
On the other hand, we could have just gotten to these restrictions in one step
using $g|_{X \times \partial V \times L}$ and
$g|_{X \times V \times \partial L}$ without passing through
$X \times \partial(V \times L) \to X \times V \times L$.
A sufficiently well-crafted fake $\partial(V \times L) \to V \times L$ should
deceive $B$ to make it into thinking these two ways of restricting and
amalgamation are the same.
\begin{equation*}
  \begin{tikzcd}
    {X \times \partial V \times \partial L} & {X \times V \times \partial L} \\
    {X \times \partial V \times L} & {X \times \partial(V \times L)} \\
    && {X \times V \times L} \\
    &&& {B}
    \arrow[from=1-1, to=1-2]
    \arrow[from=1-1, to=2-1]
    \arrow[dotted, from=1-2, to=2-2]
    \arrow[curve={height=-12pt}, from=1-2, to=3-3]
    \arrow["{g|_{X \times V \times \partial L} = (g|_{X \times \partial(V \times L)})|_{X \times V \times \partial L}}", curve={height=-30pt}, dashed, from=1-2, to=4-4]
    \arrow[dotted, from=2-1, to=2-2]
    \arrow[curve={height=12pt}, from=2-1, to=3-3]
    \arrow["{g|_{X \times \partial V \times L} = (g|_{X \times \partial(V \times L)})|_{X \times \partial V \times L}}"', curve={height=30pt}, dashed, from=2-1, to=4-4]
    \arrow[from=2-2, to=3-3]
    \arrow["g"{description}, dashed, from=3-3, to=4-4]
  \end{tikzcd}
\end{equation*}
Once again, using internal-Homs we have a very concise definition of the above
phenomenon.
\begin{definition}
  Fix maps $\partial V \to V$ and $\partial L \to L$ between exponentiable
  objects.
  A map $\partial(V \times L) \to V \times L$ between exponentiable objects
  realises a product boundary structure
  $\textcolor{darkblue0}{ [\partial V \times L, B] \times_{\partial V \times
      \partial L, E} [V \times \partial L, B] } \xrightarrow{\cong}
  \textcolor{darkblue0}{[\partial(V \times L), B]}$ on $\partial(V \times L)$
  relative to $B$ as a \emph{pushout-product approximation structure} when the
  following triangle commutes
  \begin{equation*}
    \begin{tikzcd}[cramped]
      {\textcolor{darkblue0}{[\partial V \times L, B] \times_{[\partial V \times \partial L, B]} [V \times \partial L, B]}} && {\textcolor{darkblue0}{[\partial(V \times L), B]}} \\
      & {\textcolor{darkblue0}{[V \times L, B]}}
      \arrow["\cong", from=1-1, to=1-3]
      \arrow[color=darkblue0, from=2-2, to=1-1]
      \arrow[color=darkblue0, from=2-2, to=1-3]
    \end{tikzcd}
  \end{equation*}
\end{definition}

We can now combine everything together to say when a map thinks another map
behaves like the pushout product.

\begin{definition}\label{def:pushout-product-approx}
  Given maps $\partial V \to V$ and $\partial L \to L$ and $E \to B$ where
  $\partial V, V, \partial L, L$ are exponentiable objects in $\bC$.

  A \emph{$(\partial V \to V) \ltimes (\partial L \to L)$-approximation
    structure on $\partial(V \times L) \to V \times L $ relative to $E \to B$}
  consists of product boundary approximation structures $s_E$ and $s_B$ on the
  object $\partial(V \times L)$, which has to be exponentiable, relative to $E$
  and $B$ such that
  \begin{itemize}
    \item $E \to B$ preserves the product boundary approximation structures; and
    \item $\partial(V \times L) \to V \times L$ realises $s_B$ as a
    $(\partial V \to V) \ltimes (\partial L \to L)$ approximation structure
    relative to $B$.
  \end{itemize}
  In other words, one requires black isomorphisms as below making the diagram
  commute.
  \begin{equation}\label{eqn:pushout-product-approx}\tag{\textsc{pushout-product-approx}}
    \begin{tikzcd}[cramped]
      {\textcolor{darkblue0}{[\partial V \times L, E] \times_{[\partial V \times \partial L, E]} [V \times \partial L, E]}} & {\textcolor{darkblue0}{[\partial(V \times L), E]}} \\
      {\textcolor{darkblue0}{[\partial V \times L, B] \times_{[\partial V \times \partial L, B]} [V \times \partial L, B]}} & {\textcolor{darkblue0}{[\partial(V \times L), B]}} \\
      {\textcolor{darkblue0}{[V \times L, B]}} & {\textcolor{darkblue0}{[V \times L, B]}}
      \arrow["\cong"{description}, tail reversed, from=1-1, to=1-2]
      \arrow[from=1-1, to=2-1, color=darkblue0]
      \arrow[from=1-2, to=2-2, color=darkblue0]
      \arrow["\cong"{description}, tail reversed, from=2-1, to=2-2]
      \arrow[from=3-1, to=2-1, color=darkblue0]
      \arrow["="{description}, tail reversed, from=3-1, to=3-2]
      \arrow[from=3-2, to=2-2, color=darkblue0]
    \end{tikzcd}
  \end{equation}
\end{definition}

Throughout the rest of this section, we establish analogues of constructions
from \Cref{subsec:struct-lift-rebase}.
Specifically, we construct, like in \Cref{constr:struct-lift-pb}, that
structured approximations of the pushout-product are stable under pullback.
We also construct, like in \Cref{constr:struct-lift-postcomp}, that certain
forms of structured approximations of the pushout-product are stable under the
left adjoint to the pullback functor (i.e. the post-composition map).

We start with the analogue of \Cref{constr:struct-lift-pb}.
\begin{lemma}\label{lem:pushout-product-approx-pb}
  Fix map $\phi \colon D \to C$ along with maps $\partial J \to J$ and
  $\partial V \to V$ and $E \to B$ all in $\sfrac{\bC}{C}$, where
  $\partial J, J, \partial V, V$ are exponentiable objects in $\sfrac{\bC}{C}$.

  Suppose
  \begin{equation*}
    \partial(J \times_C V) \to J \times_C V
  \end{equation*}
  has a
  $(\partial V \to V) \ltimes_C (\partial J \to J)$-approximation structure
  relative to $E \to B$ in $\sfrac{\bC}{C}$.
  Then, the pullback
  \begin{equation*}
    \partial(J \times_C V) \to \phi^*J \times_D \phi^*V
  \end{equation*}
  has a
  $(\phi^*(\partial V) \to \phi^*V) \ltimes_D (\phi^*(\partial J) \to
  \phi^*J)$-approximation structure relative to the rebased map
  $\phi^*E \to \phi^*B$ in $\sfrac{\bC}{D}$.
\end{lemma}
\begin{proof}
  This follows from the fact that the pullback functor preserves both limits and
  internal-Homs, much like in \Cref{constr:struct-lift-pb}.
\end{proof}

We also have an analogue of \Cref{constr:struct-lift-postcomp}.
\begin{lemma}\label{lem:pushout-product-approx-postcomp}
  Fix map $\phi \colon C \to D$ along with maps $\partial J \to J$ in
  $\sfrac{\bC}{C}$ and $\partial V \to V$ and $E \to B$ in $\sfrac{\bC}{D}$
  where $\partial J, J, \partial V, V$ are exponentiable in their respective
  slice categories.

  Suppose
  \begin{equation*}
    \partial(V \times_C p^*J) \to V \times_C p^*J
  \end{equation*}
  has a
  $(\partial V \to V) \ltimes_C (p^*(\partial J) \to p^*J)$-approximation structure
  relative to $p^*E \to p^*B$ in $\sfrac{\bC}{C}$.
  Then, under the $p_!$ postcomposition functor
  \begin{equation*}
    p_!(\partial(V \times_C p^*J)) \to p_!(V \times_C p^*J)
  \end{equation*}
  inherits a
  $(p_!(\partial V) \to p_!V) \ltimes_D (\partial J \to J)$-approximation
  structure relative to $E \to B$ in $\sfrac{\bC}{D}$.
\end{lemma}
\begin{proof}
  Much like \Cref{constr:struct-lift-postcomp} by using
  \Cref{lem:pushforward-pb-left-adj}.
  For example, by assumption, one has that
  $[\partial V \times_C p^*J, p^*B] \times_{[\partial V \times_C p^*(\partial
    J), p^*B]} [V \times p^*(\partial J), p^*B] \cong [\partial(V \times_C
  p^*J), p^*B]$, and the image of this isomorphism under the right adjoint
  $p_* \colon \sfrac{\bC}{C} \to \sfrac{\bC}{D}$ is
  $[p_!(\partial V) \times_D J, B] \times_{[p_!(\partial V) \times_D \partial J,
    B]} [p_!V \times \partial J, B] \cong [p_!\partial(V \times_C p^*J), B]$.
\end{proof}

We finish by showing that approximations of pushout-products are associative.
\begin{theorem}\label{thm:pushout-product-approx-assoc}
  Suppose that one has three maps $\partial U \to U$ and $\partial V \to V$ and
  $\partial W \to W$ between exponentiable objects along with two approximations
  of pushout-products $\partial(U \times V) \to U \times V$ approximating
  $(\partial U \to U) \ltimes (\partial V \to V)$ and
  $\partial(V \times W) \to V \times W$ approximating
  $(\partial V \to V) \ltimes (\partial W \to W)$ with respect to a map
  $E \to B$ between exponentiable objects.

  A map $\partial(U \times (V \times W)) \to U \times V \times W$ approximates
  the pushout-product
  \begin{align*}
    (\partial U \to U) \ltimes (\partial(V \times W) \to V \times W)
  \end{align*}
  relative to $E \to B$ if and only if it also approximates the pushout-product
  \begin{align*}
   (\partial(U \times V) \to U \times V) \ltimes (\partial W \to W)
  \end{align*}
  relative to $E \to B$.
\end{theorem}
\begin{proof}
  The $\Rightarrow$ and $\Leftarrow$ directions are symmetric.
  We show the $\Rightarrow$ implication.

  By assumption, one has horizontal isomorphisms
  \begin{equation*}
    \begin{tikzcd}
      {[\partial (U \times V) \times W, E] \times_{[\partial (U \times V) \times \partial W, E]}
        [U \times V \times \partial W, E]}
      & {[\partial((U \times V) \times W), E]}
      \\
      {[\partial (U \times V) \times W, B] \times_{[\partial (U \times V) \times \partial W, B]}
        [U \times V \times \partial W, B]}
      & {[\partial((U \times V) \times W), B]}
      \\
      {[U \times V \times W, B]} & {[U \times V \times W, B]}
      \arrow["\cong"{description}, tail reversed, from=1-1, to=1-2]
      \arrow[from=1-1, to=2-1]
      \arrow[from=1-2, to=2-2]
      \arrow["\cong"{description}, tail reversed, from=2-1, to=2-2]
      \arrow[from=3-1, to=2-1]
      \arrow[equals, from=3-1, to=3-2]
      \arrow[from=3-2, to=2-2]
    \end{tikzcd}
  \end{equation*}
  and we must show there exists dashed isomorphisms
  \begin{equation*}
    \begin{tikzcd}
      {[\partial U \times V \times W, E] \times_{[\partial U \times \partial(V \times W), E]}
        [U \times \partial(V \times W), E]}
      & {[\partial((U \times V) \times W), E]}
      \\
      {[\partial U \times V \times W, B] \times_{[\partial U \times \partial(V \times W), B]}
        [U \times \partial(V \times W), B]}
      & {[\partial((U \times V) \times W), B]}
      \\
      {[U \times V \times W, B]} & {[U \times V \times W, B]}
      \arrow["\cong"{description}, dashed, tail reversed, from=1-1, to=1-2]
      \arrow[from=1-1, to=2-1]
      \arrow[from=1-2, to=2-2]
      \arrow["\cong"{description}, dashed, tail reversed, from=2-1, to=2-2]
      \arrow[from=3-1, to=2-1]
      \arrow[equals, from=3-1, to=3-2]
      \arrow[from=3-2, to=2-2]
    \end{tikzcd}
  \end{equation*}

  We show that the left columns of the two diagrams above are isomorphic.
  Because $\partial(U \times V) \to U \times V$ approximates
  $(\partial U \to U) \ltimes (\partial V \to V)$ and
  $\partial(V \times W) \to V \times W$ approximates
  $(\partial V \to V) \ltimes (\partial W \to W)$ with respect to $E \to B$, we
  have that
  \begin{align*}
    [\partial(U \times V), B]
    &\cong
      [\partial U \times V, B] \times_{[\partial U \times \partial V, B]} [U \times \partial V, B]
    \\
    [\partial(V \times W), B]
    &\cong
      [\partial V \times W, B] \times_{[\partial V \times \partial W, B]} [V \times \partial W, B]
  \end{align*}
  Thus, applying $[X,-]$ for any $X \in \bC$, we have
  \begin{align*}
    [\partial(U \times V) \times X, B]
    &\cong
    [\partial U \times V \times X, B]
    \times_{[\partial U \times \partial V \times X, B]}
    [U \times \partial V \times X, B]
    \\
    [X \times \partial(V \times W), B]
    &\cong
      [X \times \partial V \times W, B]
      \times_{[X \times \partial V \times \partial W, B]}
      [X \times V \times \partial W, B]
  \end{align*}
  So,
  $[\partial U \times V \times W, B] \times_{[\partial U \times \partial(V \times W), B]}
  [U \times \partial(V \times W), B]$ and
  $[\partial (U \times V) \times W, B] \times_{[\partial (U \times V) \times \partial W, B]}
  [U \times V \times \partial W, B]$ are both limits of the following diagram.
  \begin{equation*}
    \begin{tikzcd}[cramped]
      {[\partial U \times V \times W, B]} & {[\partial U \times V \times \partial W, B]} & {[\partial U \times V \times \partial W, B]} & {[\partial U \times V \times W, B]} \\
      {[\partial U \times \partial V \times W, B]} & {[\partial U \times \partial V \times \partial W, B]} & {[\partial U \times V \times \partial W, B]} & {[\partial U \times \partial(V \times W), B]} \\
      {[U \times \partial V \times W, B]} & {[U \times \partial V \times \partial W, B]} & {[U \times V \times \partial W, B]} & {[U \times \partial(V \times W), B]} \\
      {[\partial(U \times V) \times W, B]} & {[\partial(U \times V) \times \partial W, B]} & {[U \times V \times \partial W, B]}
      \arrow[from=1-1, to=1-2]
      \arrow[from=1-1, to=2-1]
      \arrow[equals, from=1-2, to=1-3]
      \arrow[from=1-2, to=2-2]
      \arrow[equals, from=1-3, to=2-3]
      \arrow[from=1-4, to=2-4]
      \arrow[from=2-1, to=2-2]
      \arrow[from=2-3, to=2-2]
      \arrow[from=3-1, to=2-1]
      \arrow[from=3-1, to=3-2]
      \arrow[from=3-2, to=2-2]
      \arrow[from=3-3, to=2-3]
      \arrow[from=3-3, to=3-2]
      \arrow[from=3-4, to=2-4]
      \arrow[from=4-1, to=4-2]
      \arrow[from=4-3, to=4-2]
    \end{tikzcd}
  \end{equation*}
  So $[\partial U \times V \times W, B] \times_{[\partial U \times \partial(V \times W), B]}
  [U \times \partial(V \times W), B] \cong
  [\partial (U \times V) \times W, B] \times_{[\partial (U \times V) \times \partial W, B]}
  [U \times V \times \partial W, B]$.
  The same argument applies to show the required isomorphism with $B$ replaced
  with $E$.
\end{proof}


\section{Structured Leibniz Transpose}\label{subsec:struct-ltrans}
Recall the result from Joyal-Tierney calculus \cite[Appendix]{jt07} that given maps
$\partial V \to V$ and $\partial L \to L$, the pushout-product
$\partial V \times L \cup V \times \partial L \to V \times L$ lifts against a
map $E \to B$ exactly when $\partial V \times V$ lifts against the
pullback-power $[L,E] \to [\partial L,E] \times_{[\partial L, B]} [L,B]$.
\begin{center}
  \begin{minipage}{0.45\linewidth}
    \begin{equation*}
      \begin{tikzcd}[cramped]
        {\partial V \times L \cup V \times \partial L} & E \\
        {V \times L} & B
        \arrow[dashed, from=1-1, to=1-2]
        \arrow[from=1-1, to=2-1]
        \arrow[from=1-2, to=2-2]
        \arrow[dotted, from=2-1, to=1-2]
        \arrow[dashed, from=2-1, to=2-2]
      \end{tikzcd}
    \end{equation*}
  \end{minipage}
  \begin{minipage}{0.45\linewidth}
    \begin{equation*}
      \begin{tikzcd}[cramped]
        {\partial V} & {[L,E]} \\
        V & {[\partial L,E] \times_{[\partial L, B]} [L,B]}
        \arrow[dashed, from=1-1, to=1-2]
        \arrow[from=1-1, to=2-1]
        \arrow[from=1-2, to=2-2]
        \arrow[dotted, from=2-1, to=1-2]
        \arrow[dashed, from=2-1, to=2-2]
      \end{tikzcd}
    \end{equation*}
  \end{minipage}
\end{center}

In the previous \Cref{sec:pushout-product-approx}, we defined a notion of
approximations of pushout-products, and in \Cref{sec:struct-lift}, we defined a
notion of structured restricted lifts.
Therefore, we would like to apply these concepts for Joyal-Tierney calculus, by
replacing the pushout-product
$\partial V \times L \cup V \times \partial L \to V \times L$ with an
approximation $\partial(V \times L) \to V \times L$ and the lifting problem
above with structured restricted versions.
This arrives us at the following result.

\begin{theorem}\label{thm:struct-ltrans}
  Let there be maps $\partial V \to V$ and $\partial L \to L$ and
  $E \to \underline{E}$ and $L \to L'$ where $\partial V, V, \partial L, L, L'$
  are exponentiable objects.

  Suppose that $\partial(V \times L) \to V \times L$ structurally approximates
  the pushout-product $(\partial V \to V) \ltimes (\partial L \to L)$ relative
  to the map $E \to \ul{E}$.
  Then, we have a bijection
  \begin{equation*}
    \left(\begin{tikzcd}[cramped, column sep=small]
        \partial(V \times L) \ar[d] \\ V \times L \ar[r] & V \times L'
      \end{tikzcd}
      \squareslash
      \begin{tikzcd}[cramped, column sep=small]
        E \ar[d] \\ \underline{E}
      \end{tikzcd}\right)
    \xrightarrow{\quad\cong\quad}
    \left(\begin{tikzcd}[cramped, column sep=small]
        \partial V \ar[d] \\ V
      \end{tikzcd}
      \squareslash
      \begin{tikzcd}[cramped, column sep=small]
        & {[L,E]} \ar[d] \\
        {[\partial L, E] \times_{[\partial L, \underline{E}]} [L',\underline{E}]} \ar[r] &
        {[\partial L, E] \times_{[\partial L, \underline{E}]} [L,\underline{E}]}
      \end{tikzcd}\right)
  \end{equation*}
\end{theorem}
\begin{proof}
  We set
  $\underline{\partial}_L\set{L',E} \to \underline{\partial}_L\set{L,E}
  \leftarrow [L,E]$ as the span
  $[\partial L, E] \times_{[\partial L, \underline{E}]} [L',E] \to [\partial L,
  E] \times_{[\partial L, \underline{E}]} [L,E] \leftarrow [L,E]$ on the right
  above.
  Then, by \Cref{lem:stable-left-rep,lem:stable-right-rep}, it suffices to show
  that we have the horizontal isomorphisms as below.
  \begin{equation*}\label{eqn:ltrans}\tag{\textsc{\small ltrans}}\scriptsize
    \begin{tikzcd}[cramped, column sep=tiny]
      {\textcolor{red0}{[\partial(V \times L), E] \times_{[\partial(V \times L), \underline{E}]} [V \times L', \underline{E}]}} && {\textcolor{red0}{[\partial V, [L,E]] \times_{[\partial V, \underline{\partial}_L\set{L,E}]} [V, \underline{\partial}_L\set{L',E}]}} \\
      & {\textcolor{yellow0}{[V \times L, E]}} && {\textcolor{yellow0}{[V, [L,E]]}} \\
      {\textcolor{darkblue0}{[\partial(V \times L), E] \times_{[\partial(V \times L), \underline{E}]} [V \times L, \underline{E}]}} && {\textcolor{darkblue0}{[\partial V, [L,E]] \times_{[\partial V, \underline{\partial}_L\set{L,E}]} [V, \underline{\partial}_L\set{L,E}]}}
      \arrow["\cong"{description}, color={red0}, tail reversed, from=1-1, to=1-3]
      \arrow[color={red0}, from=1-1, to=3-1]
      \arrow[color={red0}, from=1-3, to=3-3]
      \arrow[color={yellow0}, from=2-2, to=3-1]
      \arrow[color={yellow0}, from=2-4, to=3-3]
      \arrow["\cong"{description}, color={darkblue0}, tail reversed, from=3-1, to=3-3]
      \arrow["\cong"{description, pos=0.3}, color={yellow0}, tail reversed, from=2-2, to=2-4, crossing over]
    \end{tikzcd}
  \end{equation*}

  By the assumption that $\partial(V \times L) \to V \times L$ approximates
  $(\partial V \to V) \ltimes (\partial L \to L)$ relative to
  $E \to \underline{E}$, in the diagram below, the limit of each row is isomorphic
  to the object on the right of the corresponding row; and the maps between the
  objects on the right are induced by the maps between their respective rows.
  \begin{equation*}
    \scriptsize
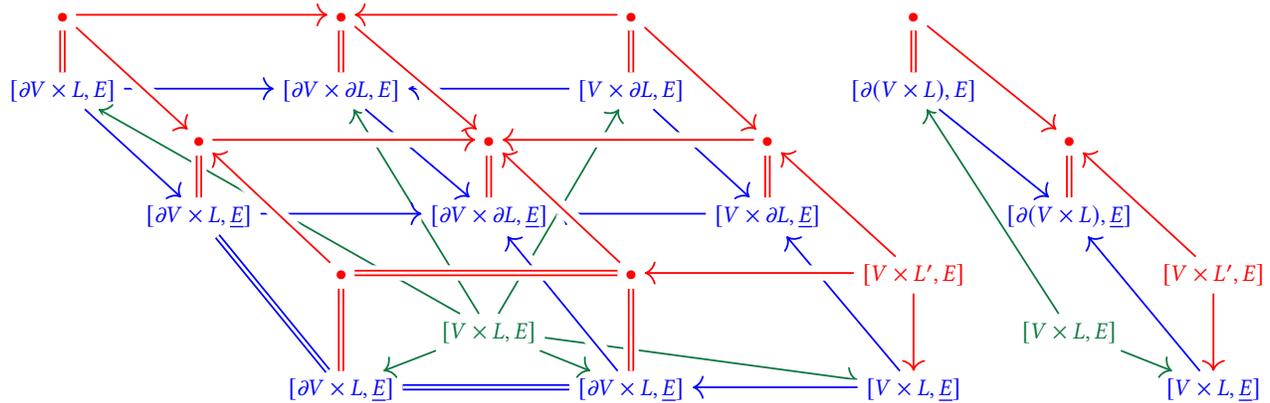
    \begin{tikzcd}[cramped, column sep=tiny, row sep=small]
      {\textcolor{red0}{\bullet}} && {\textcolor{red0}{\bullet}} && {\textcolor{red0}{\bullet}} && {\textcolor{red0}{\bullet}} \\
      \\
      {\textcolor{darkblue0}{[\partial V \times L, E]}} && {\textcolor{darkblue0}{[\partial V \times \partial L, E]}} && {\textcolor{darkblue0}{[V \times \partial L, E]}} && {\textcolor{darkblue0}{[\partial(V \times L),E]}} \\
      & {\textcolor{red0}{\bullet}} && {\textcolor{red0}{\bullet}} && {\textcolor{red0}{\bullet}} && {\textcolor{red0}{\bullet}} \\
      \\
      & {\textcolor{darkblue0}{[\partial V \times L, \underline{E}]}} && {\textcolor{darkblue0}{[\partial V \times \partial L, \underline{E}]}} && {\textcolor{darkblue0}{[V \times \partial L, \underline{E}]}} && {\textcolor{darkblue0}{[\partial(V \times L),\underline{E}]}} \\
      && {\textcolor{red0}{\bullet}} && {\textcolor{red0}{\bullet}} && {\textcolor{red0}{[V \times L', E]}} && {\textcolor{red0}{[V \times L', E]}} \\
      &&& {\textcolor{yellow0}{[V \times L, E]}} &&&& {\textcolor{yellow0}{[V \times L, E]}} \\
      && {\textcolor{darkblue0}{[\partial V \times L, \underline{E}]}} && {\textcolor{darkblue0}{[\partial V \times L, \underline{E}]}} && {\textcolor{darkblue0}{[V \times L, \underline{E}]}} && {\textcolor{darkblue0}{[V \times L, \underline{E}]}}
      \arrow[color={yellow0}, from=8-4, to=3-1]
      \arrow[color={yellow0}, from=8-4, to=3-3]
      \arrow[color={yellow0}, from=8-4, to=3-5]
      \arrow[color={yellow0}, from=8-4, to=9-3]
      \arrow[color={yellow0}, from=8-4, to=9-5]
      \arrow[color={yellow0}, from=8-4, to=9-7]
      \arrow[color={yellow0}, from=8-8, to=3-7]
      \arrow[color={yellow0}, from=8-8, to=9-9]
      \arrow[color={darkblue0}, crossing over, from=3-1, to=3-3]
      \arrow[color={darkblue0}, crossing over, from=3-1, to=6-2]
      \arrow[color={darkblue0}, crossing over, from=3-3, to=6-4]
      \arrow[color={darkblue0}, crossing over, from=3-5, to=3-3]
      \arrow[color={darkblue0}, crossing over, from=3-5, to=6-6]
      \arrow[color={darkblue0}, crossing over, from=3-7, to=6-8]
      \arrow[color={darkblue0}, crossing over, from=6-2, to=6-4]
      \arrow[color={darkblue0}, crossing over, from=6-6, to=6-4]
      \arrow[color={darkblue0}, crossing over, Rightarrow, no head, from=9-3, to=6-2]
      \arrow[color={darkblue0}, crossing over, from=9-5, to=6-4]
      \arrow[color={darkblue0}, crossing over, from=9-7, to=6-6]
      \arrow[color={darkblue0}, crossing over, from=9-9, to=6-8]
      \arrow[color={darkblue0}, crossing over, Rightarrow, no head, from=9-3, to=9-5]
      \arrow[color={darkblue0}, crossing over, from=9-7, to=9-5]
      \arrow[color={red0}, crossing over, from=1-1, to=1-3]
      \arrow[color={red0}, crossing over, Rightarrow, no head, from=1-1, to=3-1]
      \arrow[color={red0}, crossing over, from=1-1, to=4-2]
      \arrow[color={red0}, crossing over, Rightarrow, no head, from=1-3, to=3-3]
      \arrow[color={red0}, crossing over, from=1-3, to=4-4]
      \arrow[color={red0}, crossing over, from=1-5, to=1-3]
      \arrow[color={red0}, crossing over, Rightarrow, no head, from=1-5, to=3-5]
      \arrow[color={red0}, crossing over, from=1-5, to=4-6]
      \arrow[color={red0}, crossing over, Rightarrow, no head, from=1-7, to=3-7]
      \arrow[color={red0}, crossing over, from=1-7, to=4-8]
      \arrow[color={red0}, crossing over, from=4-2, to=4-4]
      \arrow[color={red0}, crossing over, Rightarrow, no head, from=4-2, to=6-2]
      \arrow[color={red0}, crossing over, Rightarrow, no head, from=4-4, to=6-4]
      \arrow[color={red0}, crossing over, from=4-6, to=4-4]
      \arrow[color={red0}, crossing over, Rightarrow, no head, from=4-6, to=6-6]
      \arrow[color={red0}, crossing over, Rightarrow, no head, from=4-8, to=6-8]
      \arrow[color={red0}, crossing over, from=7-3, to=4-2]
      \arrow[color={red0}, crossing over, Rightarrow, no head, from=7-3, to=7-5]
      \arrow[color={red0}, crossing over, Rightarrow, no head, from=7-3, to=9-3]
      \arrow[color={red0}, crossing over, from=7-5, to=4-4]
      \arrow[color={red0}, crossing over, Rightarrow, no head, from=7-5, to=9-5]
      \arrow[color={red0}, crossing over, from=7-7, to=4-6]
      \arrow[color={red0}, crossing over, from=7-7, to=7-5]
      \arrow[color={red0}, crossing over, from=7-7, to=9-7]
      \arrow[color={red0}, crossing over, from=7-9, to=4-8]
      \arrow[color={red0}, crossing over, from=7-9, to=9-9]
    \end{tikzcd}
  \end{equation*}

  Applying the exponential transpose to each object of the diagram above, one
  observes that the above diagram is isomorphic to the following diagram.
  \begin{equation*}\scriptsize
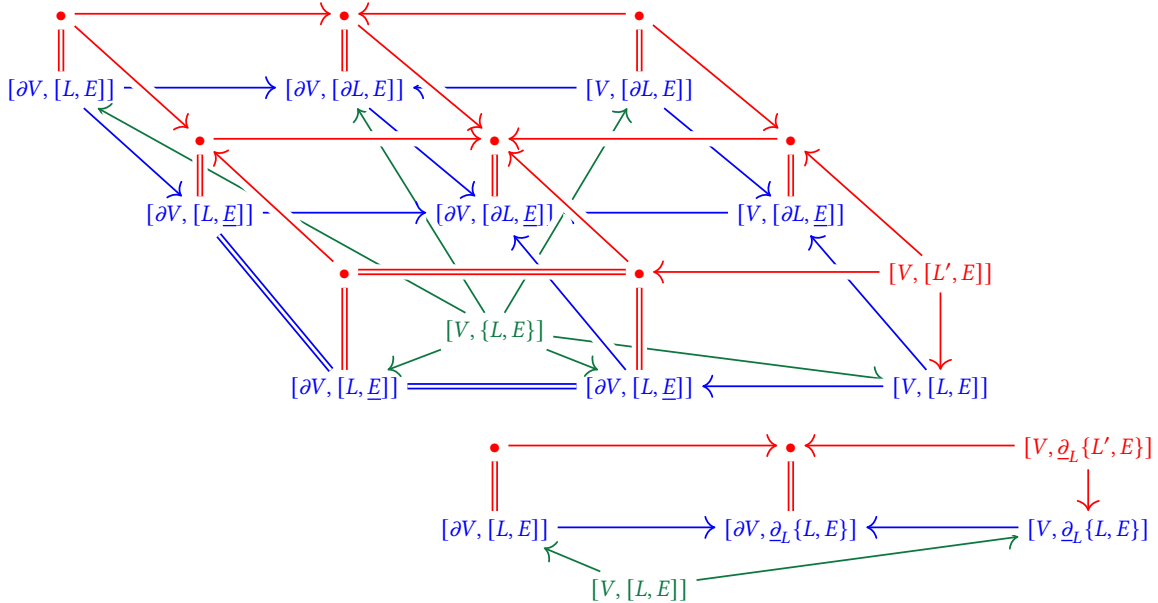
    \begin{tikzcd}[cramped, column sep=tiny, row sep=small]
      {\textcolor{red0}{\bullet}} && {\textcolor{red0}{\bullet}} && {\textcolor{red0}{\bullet}} \\
      \\
      {\textcolor{darkblue0}{[\partial V, [L, E]]}} && {\textcolor{darkblue0}{[\partial V, [\partial L, E]]}} && {\textcolor{darkblue0}{[V, [\partial L, E]]}} \\
      & {\textcolor{red0}{\bullet}} && {\textcolor{red0}{\bullet}} && {\textcolor{red0}{\bullet}} \\
      \\
      & {\textcolor{darkblue0}{[\partial V, [L, \underline{E}]]}} && {\textcolor{darkblue0}{[\partial V, [\partial L, \underline{E}]]}} && {\textcolor{darkblue0}{[V, [\partial L, \underline{E}]]}} \\
      && {\textcolor{red0}{\bullet}} && {\textcolor{red0}{\bullet}} && {\textcolor{red0}{[V, [L', E]]}} \\
      &&& {\textcolor{yellow0}{[V, \set{L, E}]}} \\
      && {\textcolor{darkblue0}{[\partial V, [L, \underline{E}]]}} && {\textcolor{darkblue0}{[\partial V, [L,\underline{E}]]}} && {\textcolor{darkblue0}{[V, [L, E]]}} \\
      &&& {\textcolor{red0}{\bullet}} && {\textcolor{red0}{\bullet}} && {\textcolor{red0}{[V, \underline{\partial}_L\set{L',E}]}} \\
      \\
      &&& {\textcolor{darkblue0}{[\partial V, [L, E]]}} && {\textcolor{darkblue0}{[\partial V, \underline{\partial}_L\set{L,E}]}} && {\textcolor{darkblue0}{[V, \underline{\partial}_L\set{L,E}]}} \\
      &&&& {\textcolor{yellow0}{[V, [L, E]]}}
      \arrow[draw={yellow0}, crossing over, from=8-4, to=3-1]
      \arrow[draw={yellow0}, crossing over, from=8-4, to=3-3]
      \arrow[draw={yellow0}, crossing over, from=8-4, to=3-5]
      \arrow[draw={yellow0}, crossing over, from=8-4, to=9-3]
      \arrow[draw={yellow0}, crossing over, from=8-4, to=9-5]
      \arrow[draw={yellow0}, crossing over, from=8-4, to=9-7]
      \arrow[draw={yellow0}, crossing over, from=13-5, to=12-4]
      \arrow[draw={yellow0}, crossing over, from=13-5, to=12-8]
      \arrow[draw={darkblue0},  crossing over, from=3-1, to=3-3]
      \arrow[draw={darkblue0},  crossing over, from=3-1, to=6-2]
      \arrow[draw={darkblue0},  crossing over, from=3-3, to=6-4]
      \arrow[draw={darkblue0},  crossing over, from=3-5, to=3-3]
      \arrow[draw={darkblue0},  crossing over, from=3-5, to=6-6]
      \arrow[draw={darkblue0},  crossing over, from=6-2, to=6-4]
      \arrow[draw={darkblue0},  crossing over, from=6-6, to=6-4]
      \arrow[draw={darkblue0},  crossing over, Rightarrow, no head, from=9-3, to=6-2]
      \arrow[color={darkblue0}, crossing over, Rightarrow, no head, from=9-3, to=9-5]
      \arrow[draw={darkblue0},  crossing over, from=9-5, to=6-4]
      \arrow[draw={darkblue0},  crossing over, from=9-7, to=6-6]
      \arrow[color={darkblue0}, crossing over, from=9-7, to=9-5]
      \arrow[draw={darkblue0},  crossing over, from=12-4, to=12-6]
      \arrow[draw={darkblue0},  crossing over, from=12-8, to=12-6]
      \arrow[draw={red0}, crossing over, from=1-1, to=1-3]
      \arrow[draw={red0}, crossing over, Rightarrow, no head, from=1-1, to=3-1]
      \arrow[draw={red0}, crossing over, from=1-1, to=4-2]
      \arrow[draw={red0}, crossing over, Rightarrow, no head, from=1-3, to=3-3]
      \arrow[draw={red0}, crossing over, from=1-3, to=4-4]
      \arrow[draw={red0}, crossing over, from=1-5, to=1-3]
      \arrow[draw={red0}, crossing over, Rightarrow, no head, from=1-5, to=3-5]
      \arrow[draw={red0}, crossing over, from=1-5, to=4-6]
      \arrow[draw={red0}, crossing over, from=4-2, to=4-4]
      \arrow[draw={red0}, crossing over, Rightarrow, no head, from=4-2, to=6-2]
      \arrow[draw={red0}, crossing over, Rightarrow, no head, from=4-4, to=6-4]
      \arrow[draw={red0}, crossing over, from=4-6, to=4-4]
      \arrow[draw={red0}, crossing over, Rightarrow, no head, from=4-6, to=6-6]
      \arrow[draw={red0}, crossing over, from=7-3, to=4-2]
      \arrow[draw={red0}, crossing over, Rightarrow, no head, from=7-3, to=7-5]
      \arrow[draw={red0}, crossing over, Rightarrow, no head, from=7-3, to=9-3]
      \arrow[draw={red0}, crossing over, from=7-5, to=4-4]
      \arrow[draw={red0}, crossing over, Rightarrow, no head, from=7-5, to=9-5]
      \arrow[draw={red0}, crossing over, from=7-7, to=4-6]
      \arrow[draw={red0}, crossing over, from=7-7, to=7-5]
      \arrow[draw={red0}, crossing over, from=7-7, to=9-7]
      \arrow[draw={red0}, crossing over, from=10-4, to=10-6]
      \arrow[draw={red0}, crossing over, Rightarrow, no head, from=10-4, to=12-4]
      \arrow[draw={red0}, crossing over, Rightarrow, no head, from=10-6, to=12-6]
      \arrow[draw={red0}, crossing over, from=10-8, to=10-6]
      \arrow[draw={red0}, crossing over, from=10-8, to=12-8]
    \end{tikzcd}
  \end{equation*}
  But the internal-Hom functor is continuous, so one observes that the limit of
  each column is isomorphic to the object on the front of the corresponding
  column; and the maps between the objects on the front are induced by the maps
  between their respective columns.

  The result then follows because the limits of the pink, blue and orange
  objects of the previous two diagrams are exactly the objects of the
  corresponding colour in \Cref{eqn:ltrans}.
  In other words, the result follows because when viewed as objects of
  $\bC^{\bullet \to \bullet \leftarrow \bullet}$ two pink-blue-orange cospans on
  the left and right of \Cref{eqn:ltrans} are both the limits of the same
  $\left(\begin{tikzcd}[cramped, row sep=small, column sep=small]
      \bullet \ar[r] \ar[d] & \bullet \ar[d] & \bullet \ar[l] \ar[d] \\
      \bullet \ar[r]        & \bullet        & \bullet \ar[l]        \\
      \bullet \ar[r] \ar[u] & \bullet \ar[u] & \bullet \ar[l] \ar[u]
    \end{tikzcd}\right)$-shaped diagram in
  $\bC^{\bullet \to \bullet \leftarrow \bullet}$.
\end{proof}

The above proof gives a procedure for transposing structure lifts.
\begin{construction}\label{constr:struct-ltrans}
  Using representability and examining the use of the exponential transpose in
  the proof of \Cref{thm:struct-ltrans}, we see that given a structured lift
  $F \in {\scriptsize\begin{tikzcd}[cramped, column sep=small, row sep=small]
      \partial(V \times L) \ar[d] \\ V \times L \ar[r] & V \times L'
  \end{tikzcd}}
  \squareslash
  {\scriptsize\begin{tikzcd}[cramped, column sep=small, row sep=small]
    E \ar[d] \\ \underline{E}
  \end{tikzcd}}$
  its transposed pointwise structured lift
  $F^\sharp \in
  {\scriptsize\begin{tikzcd}[cramped, column sep=small, row sep=small]
    \partial V \ar[d] \\ V
  \end{tikzcd}}
  \squareslash
  {\scriptsize\begin{tikzcd}[cramped, column sep=small, row sep=small]
    & {[L,E]} \ar[d] \\
    {[\partial L, E] \times_{[\partial L, \underline{E}]} [L',E]} \ar[r] &
    {[\partial L, E] \times_{[\partial L, \underline{E}]} [L,E]}
  \end{tikzcd}}$
  solves a lifting problem
  \begin{equation*}
    \begin{tikzcd}
      X \times \partial V \ar[rr, "{f}", dashed] \ar[d] && {[L,E]} \ar[d] \\
      X \times V \ar[r, "{(g,h)}"', dashed] \ar[rru, "{F_X^\sharp(f,(g,h))}"{description}, dashed]&
      {[\partial L, E] \times_{[\partial L, \underline{E}]} [L',\underline{E}]} \ar[r] &
      {[\partial L, E] \times_{[\partial L, \underline{E}]} [L,\underline{E}]}
    \end{tikzcd}
  \end{equation*}
  by the following procedure:
  \begin{enumerate}
    \item Take the exponential transposes
    $f^\dagger \colon X \times \partial V \times L \to E$ and
    $g^\dagger \colon X \times V \times \partial L \to E$ and
    $h^\dagger \colon X \times V \times L' \to E$
    \item Because $E$ believes $X \times \partial(V \times L)$ is the
    product-boundary
    $X \times \partial V \times L \cup X \times V \times \partial L$, one has a
    common extension
    $\vbr{f^\dagger,g^\dagger} \colon X \times \partial(V \times L) \to E$ of
    $f^\dagger$ and $g^\dagger$.
    \item By commutativity of the original lifting problem and the fact that
    $\underline{E}$ also thinks $X \times \partial(V \times L)$ is
    $X \times \partial V \times L \cup X \times V \times \partial L$, we can
    form the transposed lifting problem, whose commutativity stems from the fact
    that $E \to \ul{E}$ preserves the belief of product-boundary approximation.
    \begin{equation*}
      \begin{tikzcd}
        X \times \partial(V \times L)
        \ar[rr, dashed, "{\vbr{f^\dagger,g^\dagger}}"]
        \ar[d]
        &&
        E \ar[d]
        \\
        X \times V \times L
        \ar[r]
        \ar[rru, dashed, "{F_X(\vbr{f^\dagger,g^\dagger}, h^\dagger)}"{description}]
        &
        X \times V \times L'
        \ar[r, "h^\dagger"', dashed]
        &
        \underline{E}
      \end{tikzcd}
    \end{equation*}
    \item $F$ provides a solution
    $F_X(\vbr{f^\dagger,g^\dagger}, h^\dagger) \colon X \times V \times L \to E$
    to this lifting problem, whose exponential transpose is taken to be the
    solution to the original lifting problem.
    \begin{equation*}
      F_X^\sharp(f,(g,h)) \coloneqq F_X(\vbr{f^\dagger,g^\dagger},h^\dagger)^\dagger
    \end{equation*}
  \end{enumerate}
\end{construction}

We then next try to use \Cref{constr:struct-lift-right-pb} to get rid of the
restriction on the right of \Cref{thm:struct-ltrans}.
For this, we need to first compute the following pullback.
\begin{lemma}\label{lem:pullback-hom-pb}
  Let there be maps $\partial L \to L \to L'$ and $E \to \underline{E}$ where
  $\partial L, L, L'$ are exponentiable objects.
  One has the following pullback
  \begin{equation*}
    \begin{tikzcd}[cramped, column sep=small]
      {[L,E] \times_{[L,\underline{E}]} [L',\underline{E}]}
      \ar[r] \ar[d] \ar[rd, draw=none, "{\lrcorner}"{pos=0.1,scale=1.5}]
      & {[L,E]} \ar[d] \\
      {[\partial L, E] \times_{[\partial L, \underline{E}]} [L',\underline{E}]} \ar[r] &
      {[\partial L, E] \times_{[\partial L, \underline{E}]} [L,\underline{E}]}
    \end{tikzcd}
  \end{equation*}
\end{lemma}
\begin{proof}
  The objects in the above diagram are the limits of the cospans as follows
  while the maps above are the maps between the limits induced by the maps of
  the respective cospans.
  The result then follows by the fact that limits commute with limits.
  \begin{equation*}\scriptsize
    \begin{tikzcd}[column sep=small, row sep=small]
      & {[L,E]} && {[L,\underline{E}]} && {[L,\underline{E}]} && {[L,E]} \\
      {[L, E]} && {[L, \underline{E}]} && {[L', \underline{E}]} && {[L,E] \times_{[L,\underline{E}]} [L',\underline{E}]} \\
      & {[\partial L, E]} && {[\partial L, \underline{E}]} && {[L, \underline{E}]} && {[\partial L, E] \times_{[\partial L, \underline{E}]} [L,\underline{E}]} \\
      {[\partial L, E]} && {[\partial L, \underline{E}]} && {[L', \underline{E}]} && {[\partial L, E] \times_{[\partial L, \underline{E}]} [L',\underline{E}]}
      \arrow[from=1-2, to=1-4]
      \arrow[from=1-2, to=3-2]
      \arrow[from=1-4, to=3-4]
      \arrow[Rightarrow, no head, from=1-6, to=1-4]
      \arrow[Rightarrow, no head, from=1-6, to=3-6]
      \arrow[from=1-8, to=3-8]
      \arrow[Rightarrow, no head, from=2-1, to=1-2]
      \arrow[from=2-1, to=4-1]
      \arrow[Rightarrow, no head, from=2-3, to=1-4]
      \arrow[from=2-5, to=1-6]
      \arrow[from=2-7, to=1-8]
      \arrow[from=2-7, to=4-7]
      \arrow[from=3-2, to=3-4]
      \arrow[from=3-6, to=3-4]
      \arrow[Rightarrow, no head, from=4-1, to=3-2]
      \arrow[from=4-1, to=4-3]
      \arrow[Rightarrow, no head, from=4-3, to=3-4]
      \arrow[from=4-5, to=3-6]
      \arrow[from=4-5, to=4-3]
      \arrow[from=4-7, to=3-8]
      \arrow[crossing over, from=2-1, to=2-3]
      \arrow[crossing over, from=2-3, to=4-3]
      \arrow[crossing over, from=2-5, to=2-3]
      \arrow[crossing over, Rightarrow, no head, from=2-5, to=4-5]
      \arrow["\lrcorner"{anchor=center, pos=0.05, scale=1.5, rotate=45}, draw=none, from=2-1, to=3-2]
      \arrow["\lrcorner"{anchor=center, pos=0.05, scale=1.5, rotate=45}, draw=none, from=2-3, to=3-4]
      \arrow["\lrcorner"{anchor=center, pos=0.05, scale=1.5, rotate=45}, draw=none, from=2-5, to=3-6]
      \arrow["\lrcorner"{anchor=center, pos=0.05, scale=1.5, rotate=45}, draw=none, from=2-7, to=3-8]
    \end{tikzcd}
  \end{equation*}
\end{proof}

\begin{proposition}\label{prop:struct-ltrans*}
  Let there be maps $\partial V \to V$ and $\partial L \to L$ and
  $E \to \underline{E}$ and $L \to L'$ where $\partial V, V, \partial L, L, L'$
  are exponentiable objects.

  Suppose that $\partial(V \times L) \to V \times L$ structurally approximates
  $(\partial V \to V) \ltimes (\partial L \to L)$ relative to
  $E \to \underline{E}$.
  Then, we have a bijection
  \begin{equation*}
    \left(
      \begin{tikzcd}[cramped, column sep=small]
        \partial(V \times L) \ar[d] \\ V \times L \ar[r] & V \times L'
      \end{tikzcd}
      \squareslash
      \begin{tikzcd}[cramped, column sep=small]
        E \ar[d] \\ \underline{E}
      \end{tikzcd}
    \right)
    \cong
    \left(
      \begin{tikzcd}[cramped, column sep=small]
        \partial V \ar[d] \\ V
      \end{tikzcd}
      \squareslash
      \begin{tikzcd}[cramped, column sep=small]
        {[L,E] \times_{[L,\underline{E}]} [L',\underline{E}]} \ar[d] \\
        {[\partial L, E] \times_{[\partial L, \underline{E}]} [L',\underline{E}]}
      \end{tikzcd}
    \right)
  \end{equation*}
\end{proposition}
\begin{proof}
  Apply
  \Cref{thm:struct-ltrans,constr:struct-lift-right-pb,lem:pullback-hom-pb}.
\end{proof}

\begin{corollary}\label{cor:struct-ltrans*-local}
  Let there be maps $E \to \ul{E} \in \bC$ and $\partial L \to L \to L' \in \bC$
  where $\partial L, L, L'$ are exponentiable.

  Fix an object $C \in \bC$ along with a map
  $\partial V \to V \in \sfrac{\bC}{C}$ and assume that
  $\partial(L \times V) \to L \times V \in \sfrac{\bC}{C}$ approximates
  $(\partial V \to V) \ltimes_C (\partial L \times C \to L \times C) \in
  \sfrac{\bC}{C}$ relative to $E \times C \to \ul{E} \times C \in \sfrac{\bC}{C}$.
  Then,
  \begin{equation*}
    \left(
      \begin{tikzcd}[cramped, column sep=small]
        \partial(V \times L) \ar[d] \\ V \times L \ar[r] & V \times L'
      \end{tikzcd}
      \fracsquareslash{C}
      \begin{tikzcd}[cramped, column sep=small]
        E \times C \ar[d] \\ \underline{E} \times C
      \end{tikzcd}
    \right)
    \cong
    \left(
      \begin{tikzcd}[cramped, column sep=small]
        \partial V \ar[d] \\ V
      \end{tikzcd}
      \fracsquareslash{C}
      \begin{tikzcd}[cramped, column sep=small]
        {([L,E] \times_{[L,\underline{E}]} [L',\underline{E}]) \times C} \ar[d] \\
        {([\partial L, E] \times_{[\partial L, \underline{E}]} [L',\underline{E}]) \times C}
      \end{tikzcd}
    \right)
  \end{equation*}
\end{corollary}
\begin{proof}
  By \Cref{prop:struct-ltrans*}  we have that
  \begin{equation*}
    \left(
      \begin{tikzcd}[cramped, column sep=small]
        \partial(V \times L) \ar[d] \\ V \times L \ar[r] & V \times L'
      \end{tikzcd}
      \fracsquareslash{C}
      \begin{tikzcd}[cramped, column sep=small]
        E \times C \ar[d] \\ \underline{E} \times C
      \end{tikzcd}
    \right)
    \cong
    \left(
      \begin{tikzcd}[cramped, column sep=small]
        \partial V \ar[d] \\ V
      \end{tikzcd}
      \fracsquareslash{C}
      \begin{tikzcd}[cramped, column sep=small]
        {[L \times C,E \times C]_C
          \times_{[L \times C,\underline{E} \times C]_C}
          [L' \times C,\underline{E} \times C]_C} \ar[d] \\
        {[\partial L \times C, E \times C]_C
          \times_{[\partial L \times C, \underline{E} \times C]_C}
          [L' \times C,\underline{E} \times C]_C}
      \end{tikzcd}
    \right)
  \end{equation*}
  The result then follows from the fact that limits commute with limits and
  pullbacks preserve exponentials.
\end{proof}

\subsection*{Structured Leibniz Transposes for Path Objects}
We now focus on a specific case of structured Leibniz transposes involving
interval and path objects.
In cofibrantly-generated model structures, one often makes use of an interval
object to generate the trivial cofibrations.
In particular, an interval object $I$ is some good cylinder object for the
terminal object and the trivial cofibrations are usually generated in such a way
that the pushout-product of the interval boundary inclusion
$\partial I \hookrightarrow I$ with a trivial cofibration $K \to L$ is remains a
trivial cofibration.
Then, by the Leibniz transpose operation, this allows one to construct fibred
path objects as the pullback-power.

In model categories, being a fibration requires the mere existence of lifting
solutions and one has access to colimiting concepts such as pushout-product.
However, in our setting, we wish to use the framework developed in
\Cref{sec:struct-lift} to enforce some sort of uniformity on the lifts, and we
wish to use the framework developed in \Cref{sec:pushout-product-approx} to
reduce dependency on colimiting concepts for applicability in type-theoretic
settings.

Therefore, to reproduce the path object construction via the Leibiniz transpose
in a structured manner amenable to type-theoretic applications, the idea is to
use \Cref{prop:struct-ltrans*} and the associated
\Cref{cor:struct-ltrans*-local} to phrase this structured version of the path
object construction by way of the Leibiniz transpose involving the interval
object.
The formal statement of this is expressed in \Cref{cor:struct-ltrans-path}.

However, to build up to this, we need to first axiomatise the fibred path object
constructed from an interval object.
\begin{definition}\label{def:I-path-obj}
  Let $\partial I \to I \in \bC$ be a map between exponentiable objects.
  For each $X$, we denote its \emph{$I$-path object} as the exponential
  $P_\bC^I(X) \coloneqq [I,X]$ and the \emph{coboundary of the $I$-path object} is
  the exponential $\partial P_\bC^I(X) \coloneqq [\partial I, X]$ so that there is a
  \emph{endpoint evaluation map
    $\ev_\partial \colon P_\bC^I(X) \to \partial P_\bC^I(X)$}.

  Given a map $E \to B \in \bC$, we also write
  $\hP_\bC^I(E \to B) \colon P_\bC^I(E) \to \partial P_\bC^I(E) \times_{\partial
    P_\bC^I(B) } P_\bC^I(B)$ for the pullback-Hom of $\partial I \to I$ with
  $E \to B$.

  For each $E \to B \in \sfrac{\bC}{C}$ and a map $\partial I \to I \in \bC$ we
  write $P_C^I(E)$ and $\partial P_C^I(E)$ and $\hP_C^I(E \to B)$ to mean the
  respective constructions $P_{\bC/C}^{I \times C}(E)$ and
  $\partial P^{I \times C}_{\bC/C}(E)$ and $\hP^{I \times C}_{\bC/C}(E \to B)$
  in the slice category.
\end{definition}

We observe that the local exponential used in defining $I$-path objects in
slices occur as pullbacks of $I$-path objects in the ambient category.
This follows from the following more general fact.
\begin{lemma}\label{lem:local-exp-pb}
  For any $E \to B \in \sfrac{\bC}{B}$, and an exponentiable $A \in \bC$, the
  local exponential $[A \times B, E]_B \cong B \times_{[A,B]} [A,E]$ is just the
  pullback-Hom of $A \to 1$ with $E \to B$ in $\bC$.
  \begin{equation*}
    \begin{tikzcd}[cramped]
      {[A \times B, E]_B} & {[A,E]} \\
      B & {[A,B]}
      \arrow[from=1-1, to=1-2]
      \arrow[from=1-1, to=2-1]
      \arrow["\lrcorner"{anchor=center, pos=0.15, scale=1.5}, draw=none, from=1-1, to=2-2]
      \arrow[from=1-2, to=2-2]
      \arrow[from=2-1, to=2-2]
    \end{tikzcd}
  \end{equation*}
\end{lemma}
\begin{proof}
  We argue using representability, by showing that
  $[A \times B, E]_B \to B \in \sfrac{\bC}{B}$ and
  $B \times_{[A,B]} [A,E] \to B \in \sfrac{\bC}{B}$ represent the same presheaf
  over $\sfrac{\bC}{B}$.

  We notice that, in the diagram on the left, for each $x \colon X \to B \in \sfrac{\bC}{B}$, the local
  Hom-set $\sfrac{\bC}{B}(A \times X, E)$ are precisely those maps
  $A \times X \to E$ whose post-composition with $E \to B$ is exactly
  $A \times X \xrightarrow{\proj_2} X \xrightarrow{x} B$.
  Therefore, we have the pullback as on the right.
  \begin{center}
    \begin{minipage}{0.2\linewidth}
      \begin{equation*}\small
        \begin{tikzcd}[cramped, row sep=small, column sep=small]
          {A \times X} & {A \times B} & E \\
          X & B
          \arrow[from=1-1, to=1-2]
          \arrow[curve={height=-12pt}, dashed, from=1-1, to=1-3]
          \arrow[from=1-1, to=2-1]
          \arrow[from=1-2, to=2-2]
          \arrow[from=1-3, to=2-2]
          \arrow[from=2-1, to=2-2]
        \end{tikzcd}
      \end{equation*}
    \end{minipage}
    \begin{minipage}{0.7\linewidth}
      \begin{equation*}\small
        \begin{tikzcd}[cramped, row sep=small, column sep=small]
          {\sfrac{\bC}{B}(A \times X, E)} & {\bC(X, B \times_{[A, B]} [A, E])} & {\bC(A \times X, E)} & {\bC(X, [A,E])} \\
          1 & {\bC(X,B)} & {\bC(A \times X,B)} & {\bC(X, [A,B])}
          \arrow[from=1-1, to=1-2]
          \arrow[from=1-1, to=2-1]
          \arrow[from=1-2, to=1-3]
          \arrow[from=1-2, to=2-2]
          \arrow["\cong", from=1-3, to=1-4]
          \arrow[from=1-3, to=2-3]
          \arrow[from=1-4, to=2-4]
          \arrow["x"', from=2-1, to=2-2]
          \arrow[from=2-2, to=2-3]
          \arrow["\cong"', from=2-3, to=2-4]
          \arrow["\lrcorner"{anchor=center, pos=0.15, scale=1.5}, draw=none, from=1-1, to=2-2]
          \arrow["\lrcorner"{anchor=center, pos=0.15, scale=1.5}, draw=none, from=1-2, to=2-3]
          \arrow["\lrcorner"{anchor=center, pos=0.15, scale=1.5}, draw=none, from=1-3, to=2-4]
        \end{tikzcd}
      \end{equation*}
    \end{minipage}
  \end{center}
  But analogously, the fibre of $\bC(X, B \times_{[A, B]} [A, E]) \to \bC(X,B)$
  over $x$ is exactly $\sfrac{\bC}{B}(X, B \times_{[A, B]} [A, E])$ as well.
  This proves
  $\sfrac{\bC}{B}(A \times X, E) \cong \sfrac{\bC}{B}(X, B \times_{[A, B]} [A,
  E])$ so that $[A \times B, E]_B \cong B \times_{[A \times B]} [A, E]$.
\end{proof}

As from the name, we wish to take the $I$-path objects as a path object
construction in slices.
They indeed serve this goal in model categories when $I$ is chosen to be an
interval.
Our notion of an interval is made precise using the notion of good cylinder
objects, which we now recall.
\begin{definition}
  In a model category $\bC$, a \emph{good} cylinder object for an object $A$ is
  the middle object $\textsf{Cyl}(A)$ of any factorisation of the codiagonal
  $A \sqcup A \to A$ as
  $A \sqcup A \hookrightarrow \textsf{Cyl}(A) \xrightarrow{\sim} A$ into a
  cofibration followed by a weak equivalence.
  It is \emph{very good} when $\textsf{Cyl}(A) \xrightarrow{\sim} A$ is a
  trivial fibration.
\end{definition}
\begin{theorem}\label{thm:fibred-path-fibration}
  Let $\bC$ be locally cartesian closed and equipped with a model structure in
  which $I$ is a good, but not necessarily very good, cylinder object for the
  terminal object and the pullback-power of a cofibration with a fibration is a
  fibration.
  %

  Then, by taking $\partial I \coloneqq 1 \sqcup 1$, for any $E \twoheadrightarrow B \in \sfrac{\bC}{\underline{B}}$
  fibration, the pullback-Hom
  $\hP_{\ul{B}}^{I}(E \to B) \colon P_{\ul{B}}^{I}(E)
  \twoheadrightarrow
  \partial P_{\ul{B}}^{I}(E) \times_{\partial P_{\ul{B}}^I(E) } \partial P_{\ul{B}}^{I}(B)$
  is a fibration.
\end{theorem}
\begin{proof}
  %
  %
  Using \Cref{lem:local-exp-pb} four times and yet another 3-by-3 argument
  \begin{equation*}
    \begin{tikzcd}[cramped, row sep=tiny, column sep=tiny]
      {\ul{B}} &&& {\ul{B}} \\
      && {[I,\ul{B}]} &&& {[\partial I, \ul{B}]} \\
      &&&& {[I,E]} &&& {[\partial I, E]} \\
      {\ul{B}} &&& {\ul{B}} &&& {P_{\ul{B}}^I(E)} &&& {\partial P_{\ul{B}}^I(E)} \\
      && {[I,\ul{B}]} &&& {[\partial I, \ul{B}]} \\
      &&&& {[I,B]} &&& {[\partial I, B]} \\
      &&&&&& {P_{\ul{B}}^I(B)} &&& {\partial P_{\ul{B}}^I(B)}
      \arrow[from=1-1, to=1-4]
      \arrow[from=1-1, to=2-3]
      \arrow[from=1-1, to=4-1]
      \arrow[from=1-4, to=2-6]
      \arrow[from=1-4, to=4-4]
      \arrow[from=2-6, to=5-6]
      \arrow[from=3-8, to=2-6]
      \arrow[from=3-8, to=6-8]
      \arrow[from=4-1, to=4-4]
      \arrow[from=4-1, to=5-3]
      \arrow[from=4-4, to=5-6]
      \arrow[from=4-10, to=7-10]
      \arrow[from=5-3, to=5-6]
      \arrow[from=6-5, to=5-3]
      \arrow[from=6-5, to=6-8]
      \arrow[from=6-8, to=5-6]
      \arrow[from=7-7, to=7-10]
      \arrow[crossing over, from=2-3, to=2-6]
      \arrow[crossing over, from=2-3, to=5-3]
      \arrow[crossing over, from=3-5, to=2-3]
      \arrow[crossing over, from=3-5, to=3-8]
      \arrow[crossing over, from=3-5, to=6-5]
      \arrow[crossing over, from=4-7, to=4-10]
      \arrow[crossing over, from=4-7, to=7-7]
    \end{tikzcd}
  \end{equation*}
  we see that the pullback of the front face
  $\partial P_{\ul{B}}^I(E) \times_{\partial P_{\ul{B}}^I(E) } \partial
  P_{\ul{B}}^I(B)$ can be computed by first taking the pullbacks of each layer
  $\ul{B} \to [I,\ul{B}] \leftarrow [I,B] \times_{[\partial I, B]} [\partial I,
  E]$ then taking the limit.
  Thus, the map
  $\hP_{\ul{B}}^I(E \to B) \colon P_{\ul{B}}^I(E) \twoheadrightarrow \partial
  P_{\ul{B}}^I(E) \times_{\partial P_{\ul{B}}^I(E) } \partial P_{\ul{B}}^I(B)$
  fits into the following cube whose top, bottom and right face are pullbacks.
  \begin{equation*}
    \begin{tikzcd}[cramped, row sep=small, column sep=small]
      {P_{\ul{B}}^I(E)} && {\ul{B}} \\
      & {[I,E]} && {[I,\ul{B}]} \\
      {\partial P_{\ul{B}}^I(E) \times_{\partial P_{\ul{B}}^I(E) } \partial   P_{\ul{B}}^I(B)} && {\ul{B}} \\
      & {[I,B] \times_{[\partial I, B]} [\partial I, E]} && {[I, \ul{B}]}
      \arrow[from=1-1, to=1-3]
      \arrow[from=1-1, to=2-2]
      \arrow[from=1-1, to=3-1]
      \arrow[from=1-3, to=2-4]
      \arrow[from=1-3, to=3-3]
      \arrow[from=2-4, to=4-4]
      \arrow[from=3-1, to=3-3]
      \arrow[from=3-1, to=4-2]
      \arrow[from=3-3, to=4-4]
      \arrow[from=4-2, to=4-4]
      \arrow[crossing over, from=2-2, to=2-4]
      \arrow[crossing over, from=2-2, to=4-2]
      \arrow["\lrcorner"{anchor=center, pos=0.15, scale=1.5}, draw=none, from=1-1, to=2-4]
      \arrow["\lrcorner"{anchor=center, pos=0.15, scale=1.5}, draw=none, from=1-1, to=4-2]
      \arrow["\lrcorner"{anchor=center, pos=0.15, scale=1.5}, draw=none, from=1-3, to=4-4]
      \arrow["\lrcorner"{anchor=center, pos=0.15, scale=1.5}, draw=none, from=3-1, to=4-4]
    \end{tikzcd}
  \end{equation*}
  This means that, $\hP_{\ul{B}}^I(E \to B)$ occurs as a pullback of
  $[I,E] \to [I,B] \times_{[\partial I, B]} [\partial I, E]$.
  The result then follows by the pullback-power assumption.
  %
  %
  %
\end{proof}

In view of \Cref{thm:fibred-path-fibration}, we now axiomatise the structures of
an interval object so that \Cref{thm:fibred-path-fibration} can be reproduced in
a type-theoretic setting in \Cref{cor:struct-ltrans-path}.
\begin{definition}\label{def:interval-obj}
  A \emph{pre-interval} structure on an exponentiable object $I$ consists of two
  points $\set{0}, \set{1} \rightrightarrows I$ along with a map
  $\partial I \hookrightarrow I$ between exponentiable objects such that for the
  pullback $\set{0} \cap \set{1}$ defined as the pullback on the left, one has a
  factorisation on the right.
  \begin{center}
    \begin{minipage}{0.45\linewidth}
      \begin{equation*}
        \begin{tikzcd}
          \set{0} \cap \set{1}
          \ar[r, hook] \ar[d, hook] \ar[rd, "{\lrcorner}"{pos=0.125}, draw=none]
          & \set{0} \ar[d, hook]
          \\
          \set{1} \ar[r, hook] & I
        \end{tikzcd}
      \end{equation*}
    \end{minipage}
    \begin{minipage}{0.45\linewidth}
      \begin{equation*}
        \begin{tikzcd}
          \set{0} \cap \set{1}
          \ar[r, hook] \ar[d, hook]
          & \set{0} \ar[d, hook] \ar[rdd, bend left=20, hook]
          \\
          \set{1} \ar[drr, bend right=20, hook] \ar[r, hook] & \partial I \ar[rd, hook]
          \\
          & & I
        \end{tikzcd}
      \end{equation*}
    \end{minipage}
  \end{center}

  A pre-interval structure on an object $I$ becomes an \emph{interval structure
    relative to a set of objects $\McB$} when for each $B \in \McB$, one
  further has that $[\set{0} \cap \set{1}, B] \cong 1$ and that
  $B \cong [\set{0}, B] \to [\partial I, B] \leftarrow [\set{1}, B] \cong B$
  is a product span (i.e. $[\partial I, B] \cong B \times B$).
\end{definition}

\begin{corollary}\label{cor:struct-ltrans-path}
  Fix a map $E \to B$ along with an exponentiable object $I$ equipped with an
  interval structure $(\set{0},\set{1} \rightrightarrows I, \partial I \to I)$
  relative to $\set{E,B}$.

  Further take maps $\partial V \to V \in \sfrac{\bC}{C}$ between exponentiable
  objects in a slice and a map
  $\partial(V \times I) \to V \times I \in \sfrac{\bC}{C}$ between exponentiable
  objects that structurally approximates
  $(\partial V \to V) \ltimes_C (\partial I \times C \to I \times C)$ relative
  to the map $E \times C \to B \times C \in \sfrac{\bC}{C}$.
  Then, we have a bijection
  \begin{equation*}
    \left(
    \begin{tikzcd}[cramped, column sep=small]
      \partial(V \times I) \ar[d] \\ V \times I \ar[r] & V
    \end{tikzcd}
    \fracsquareslash{C}
    \begin{tikzcd}[cramped, column sep=small]
      E \times C \ar[d] \\ B \times C
    \end{tikzcd}
    \right)
    \cong
    \left(
    \begin{tikzcd}[cramped, column sep=small]
      \partial V \ar[d] \\ V
    \end{tikzcd}
    \fracsquareslash{C}
    \begin{tikzcd}[cramped, column sep=small]
      {P^I_B(E) \times C} \ar[d] \\
      {(E \times_{B} E) \times C}
    \end{tikzcd}
    \right)
  \end{equation*}
\end{corollary}
\begin{proof}
  Immediate by \Cref{cor:struct-ltrans*-local,lem:local-exp-pb}.
\end{proof}

In anticipation of working with slice categories, we note several base change
properties for interval structures.
These are analogues of
\Cref{lem:pushout-product-approx-pb,constr:struct-lift-pb} but for interval and
path objects.


\begin{lemma}\label{lem:interval-obj-fibred-path-pb}
  Let there be a map $\phi \colon D \to C$.
  Then, a pre-interval structure
  $(\set{0}, \set{1} \rightrightarrows I, \partial I \to I)$ on an exponentiable
  object $I \to C \in \sfrac{\bC}{C}$ gives rise to a pre-interval structure
  $(\phi^*\set{0}, \phi^*\set{1} \rightrightarrows \phi^*I, \phi^*(\partial I)
  \to \phi^*I)$ on $\phi^*I \to D \in \sfrac{\bC}{C}$.

  If further this pre-interval structure on $I \to C$ is an interval structure
  relative to an object $B \to C \in \sfrac{\bC}{C}$ then the pre-interval
  structure on $\phi^*I \to D$ is also an interval structure relative to
  $\phi^*B \to D$.

  Likewise, for maps $E \to B, \partial I \to I \in \sfrac{\bC}{C}$, the image
  under the pullback functor $\phi^* \colon \sfrac{\bC}{C} \to \sfrac{\bC}{D}$
  of the fibred path object, the endpoint evaluation maps and the pullback-Hom
  constructed from $\partial I \to I \in \sfrac{\bC}{C}$ as on the left column
  are respectively isomorphic to the fibred path object, the endpoint evaluation
  maps and the pullback-Hom constructed from the interval object
  $\phi^*(\partial I) \to \phi^*I \in \sfrac{\bC}{D}$ as on the right column.
  \begin{equation*}\small
    \begin{tikzcd}[cramped, column sep=0.1em, row sep=0.25em, execute at end picture={
        \begin{scope}[on background layer]
          \draw[->] (C) -- (D);
        \end{scope}}]
      |[alias=C]| {\sfrac{\bC}{C}} & \colorbox{bgcolor}{$\phi^*$} & |[alias=D]| {\sfrac{\bC}{D}} \\
      {P_{\bC/C}^I(E)} & \mapsto & {P_{\bC/D}^{\phi^*I}(\phi^*E)} \\
      {P_{\bC/C}^I(B)} & \mapsto & {P_{\bC/D}^{\phi^*I}(\phi^*B)} \\
      {\left(P_{\bC/C}^I(E) \to \partial P_{\bC/C}^I(E)\right)} & \mapsto & {\left(P_{\bC/D}^{\phi^*I}(E) \to \partial P_{\bC/D}^{\phi^*I}(E)\right)} \\
      {\left(P_{\bC/C}^I(B) \to \partial P_{\bC/C}^I(B)\right)} & \mapsto & {\left(P_{\bC/D}^{\phi^*I}(B) \to \partial P_{\bC/D}^{\phi^*I}(B)\right)} \\
      {\scriptstyle \left(P_{\bC/C}^I(B) \xrightarrow{\hP_{\bC/C}^I(E \to B)} \partial P^I_{\bC/C}(E) \times_{\partial P^I_{\bC/C}(E)} P^I_{\bC/C}(B)\right)} & \mapsto & {\scriptstyle \left(P_{\bC/D}^{\phi^*I}(\phi^*B) \xrightarrow{\hP_{\bC/C}^{\phi^*I}(\phi^*E \to \phi^*B)} \partial P^{\phi^*I}_{\bC/D}(\phi^*E) \times_{\partial P^{\phi^*I}_{\bC/D}(\phi^*E)} P^{\phi^*I}_{\bC/D}(\phi^*B)\right)}
    \end{tikzcd}
  \end{equation*}
\end{lemma}
\begin{proof}
  Straightforward by using the fact that the pullback functor preserves
  internal-Homs and limits.
\end{proof}


\section{Relating Lifts}\label{sec:relating-lifts}
In orthogonal factorisation systems, lifts are unique and in model
categories, these lifts are homotopically unique.
In other words, two lifts to the same lifting problem in these systems are
related in some way: by the equality relation in the former case and by the
homotopy equivalence relation in the latter.
On the other hand, in \Cref{sec:struct-lift}, we introduced the concept of
lifting structures as specific choices of lifts to lifting problems satisfying
some uniformity conditions.
Therefore, we would like to combine these two concepts of relating two lifts and
uniform choices of lifts to talk about two uniform lifts being uniformly related
to each other.
In particular, we show in \Cref{thm:lift-htpy-unique} that in certain model
categories, the homotopies involved in the homotopy uniqueness of lifts can be
chosen in a suitably uniform manner.

In order to do so, we first fix an abstraction of a ``relation'' and a
``witness'' of said relation, which is an internalisation of a binary relation
to slice categories.

\begin{definition}\label{def:fibrewise-rel}
  Let there be a map $E \to B$.
  A \emph{fibrewise relation} is a map $R_{B}(E) \to E \times_{B} E$.
  Given $C \to B$ and two maps $f_1,f_2 \colon C \to E$ over
  $B$, we say \emph{$H$ witnesses that $f_1$ and $f_2$ agree up to
    relation $R_{B}(E)$} when $H \colon C \to R_{B}(E)$
  is a factorisation of $(f_1,f_2) \colon C \to E \times_{B} E$ via
  $R_{B}(E) \to E \times_{B} E$.
  \begin{equation*}
    \begin{tikzcd}[cramped, row sep=small]
      &&& E \\
      C && {R_{B}(E)} && {E \times_{B} E} \\
      && E \\
      &&& {B}
      \arrow[from=1-4, to=4-4]
      \arrow["{f_1}", from=2-1, to=1-4, dashed]
      \arrow["H"{description}, dashed, from=2-1, to=2-3]
      \arrow["{f_2}"{description}, from=2-1, to=3-3, dashed]
      \arrow[curve={height=12pt}, from=2-1, to=4-4]
      \arrow[from=2-5, to=1-4]
      \arrow[from=2-5, to=4-4]
      \arrow[from=3-3, to=4-4]
      \arrow[crossing over, from=2-3, to=2-5]
      \arrow[crossing over, from=2-5, to=3-3]
    \end{tikzcd}
  \end{equation*}
\end{definition}

\begin{example}\label{ex:fibrewise-rel}
  Our two main sources of fibrewise relation are those giving rise to an equality
  and a right homotopy equivalence.
  \begin{enumerate}
    \item If one takes $R_B(E)$ to be the diagonal
    $E \hookrightarrow E \times_B E$ witnesses $H$ are unique and $f_1$ agrees
    with $f_2$ precisely when they are exactly equal.
    \item In model categories, if $R_B(E)$ is taken to be the path object of $E$
    as an object of the slice over $B$ then a witness $H$ is precisely a right
    homotopy between $f_1$ and $f_2$.
  \end{enumerate}
\end{example}

We will now try to motivate how these fibrewise relations are used to uniformly
relate lifts.
Consider a map $U \to V$ and a map of spans $(e,b,b')$ and two lifting
structures $F_i$ for $i=1,2$
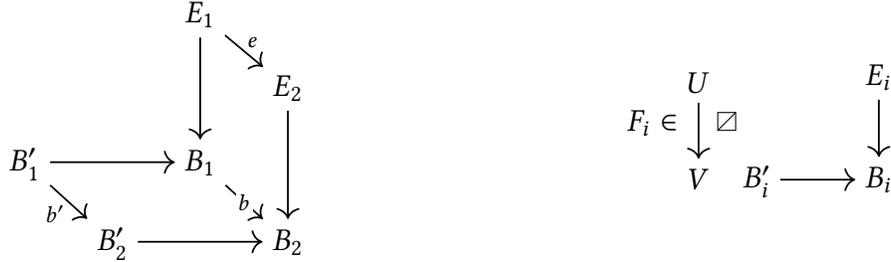
\begin{center}
  \begin{minipage}{0.45\linewidth}
    \begin{equation*}
      \begin{tikzcd}[cramped, row sep=small, column sep=small]
        && {E_1} \\
        &&& {E_2} \\
        {B_1'} && {B_1} \\
        & {B_2'} && {B_2}
        \arrow["e", from=1-3, to=2-4]
        \arrow[from=1-3, to=3-3]
        \arrow[from=2-4, to=4-4]
        \arrow[from=3-1, to=3-3]
        \arrow["{b'}"', from=3-1, to=4-2]
        \arrow["{b}"{description}, from=3-3, to=4-4]
        \arrow[from=4-2, to=4-4]
      \end{tikzcd}
    \end{equation*}
  \end{minipage}
  \begin{minipage}{0.45\linewidth}
    \begin{equation*}
      F_i \in {\begin{tikzcd}[cramped, column sep=small, row sep=small]
          U \ar[dd] \\ \\ V
        \end{tikzcd}}
      \squareslash
      {\begin{tikzcd}[cramped, column sep=small, row sep=small]
          && E_i \ar[dd] \\ \\ B_i' \ar[rr] & & B_i
        \end{tikzcd}}
    \end{equation*}
  \end{minipage}
\end{center}
Then, for $X \in \bC$ fixed, given two lifting problems $(u,v)$ of
$X \times U \to X \times V$ against $E_1 \to B_1$ restricted along
$B_1' \to B_1$ as below, one can produce a solution $F_1(u,v)$.
On the other hand, one can also induce a new lifting problem $(eu, b'v)$ against
$E_2 \to B_2$ restricted along $B_2' \to B_2$ and get a solution $F_2(eu, b'v)$.
This gives two maps
$u \cdot F_1(u,v), F_2(eu, b'v) \colon X \times U \rightrightarrows E_2$ over
$B_2$ as on the left below, where the triangle involving the lifts and the map
$e$ does not necessarily commute.
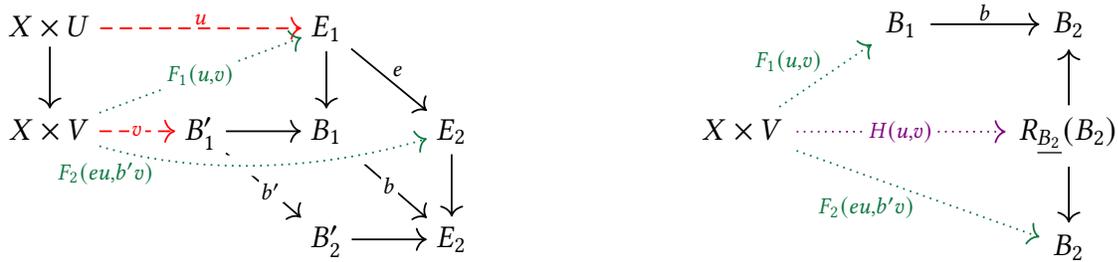
\begin{center}
  \begin{minipage}{0.45\linewidth}
    \begin{tikzcd}[cramped]
      {X \times U} && {E_1} \\
      {X \times V} & {B_1'} & {B_1} & {E_2} \\
      && {B_2'} & {E_2}
      \arrow["u", from=1-1, to=1-3, dashed, color=red0]
      \arrow[from=1-1, to=2-1]
      \arrow[from=1-3, to=2-3]
      \arrow["e", from=1-3, to=2-4]
      \arrow["{F_1(u, v)}"{description}, dotted, from=2-1, to=1-3, color=yellow0]
      \arrow["{v}"{description}, from=2-1, to=2-2, dashed, color=red0]
      \arrow[from=2-2, to=2-3]
      \arrow["{b'}"{description, pos=0.6}, from=2-2, to=3-3]
      \arrow["{b}"{description}, from=2-3, to=3-4]
      \arrow[from=2-4, to=3-4]
      \arrow[from=3-3, to=3-4]
      \arrow["{F_2(eu,b'v)}"'{pos=0.2}, curve={height=16pt}, dotted, from=2-1, to=2-4, crossing over, color=yellow0]
    \end{tikzcd}
  \end{minipage}
  \begin{minipage}{0.45\linewidth}
    \begin{equation*}
      \begin{tikzcd}
        & {B_1} & {B_2} \\
        {X \times V} && {R_{\underline{B_2}}(B_2)} \\
        && {B_2}
        \arrow["b", from=1-2, to=1-3]
        \arrow["{F_1(u,v)}", from=2-1, to=1-2, dotted, color=yellow0]
        \arrow["{H(u,v)}"{description}, from=2-1, to=2-3, dotted, color=magenta0]
        \arrow["{F_2(eu,b'v)}"', from=2-1, to=3-3, dotted, color=yellow0]
        \arrow[from=2-3, to=1-3]
        \arrow[from=2-3, to=3-3]
      \end{tikzcd}
    \end{equation*}
  \end{minipage}
\end{center}
A witness $H(u,v)$ that $u \cdot F_1(u,v)$ and $F_2(eu, b'v)$ are related up to
relation $R_{B_2}(E_2) \to E_2 \times_{B_2} E_2$ is then a factorisation $H$ as
on the right above.

Instantiating with \Cref{ex:fibrewise-rel}, we see that if
$R_{B_2}(E_2) \to E_2 \times_{B_2} E_2$ is the diagonal then such a map $H(u,v)$
exists precisely when $e \cdot F_1(u,v) = F_2(eu,b'v)$ and if
$R_{B_2}(E_2) \to E_2 \times_{B_2} E_2$ is the fibred path object then $H(u,v)$
is a right homotopy $H(u,v) \colon e \cdot F_1(u,v) \simeq F_2(eu,b'v)$ over
$E_2$.
As a further special case, if $(e,b,b')$ are taken to be all identities, then
$F_1,F_2$ both solve the same lifting problems, so the above two cases reduce
respectively to saying $F_1,F_2$ produce the same lifting solution or
right homotopically identical lifting solutions.

Now fix another map $t \colon Y \to X$ so that uniformity condition gives
\begin{align*}
  F_1(u,v) \cdot (t \times V) &= F_1(u \cdot (t \times U), v \cdot (t \times V)) \\
  F_2(eu,b'v) \cdot (t \times V) &= F_2(eu \cdot (t \times U), b'v \cdot (t \times V))
\end{align*}
Then, using $H(u,v)$, one can produce a witness $H(u,v) \cdot (t \times V)$
saying that $e \cdot F_1(u \cdot (t \times U), v \cdot (t \times V))$ and
$F_2(eu \cdot (t \times U), b'v \cdot (t \times V)$.

Therefore, if we let $(u,v)$ vary and associate a family of witnesses $H(u,v)$,
a natural uniformity condition is to require that
\begin{equation*}
  H(u,v) \cdot (t \times V) = H(u \cdot (t \times U), v \cdot (t \times V))
\end{equation*}
In view of the internalisation results of \Cref{lem:stable-right-rep}, we see
that the uniformity condition of witnesses can be formalised as follows using
the locally cartesian closed structure.

\begin{definition}\label{def:struct-lift-rel}
  Fix a map $U \to V$ between exponentiable objects and $(e,b,b')$ a map of
  spans and two lifting structures $F_i$ for $i=1,2$
  \begin{center}
    \begin{minipage}{0.45\linewidth}
      \begin{equation*}
        \begin{tikzcd}[cramped, row sep=small, column sep=small]
          && {E_1} \\
          &&& {E_2} \\
          {B_1'} && {B_1} \\
          & {B_2'} && {B_2}
          \arrow["e", from=1-3, to=2-4]
          \arrow[from=1-3, to=3-3]
          \arrow[from=2-4, to=4-4]
          \arrow[from=3-1, to=3-3]
          \arrow["{b'}"', from=3-1, to=4-2]
          \arrow["{b}"{description}, from=3-3, to=4-4]
          \arrow[from=4-2, to=4-4]
        \end{tikzcd}
      \end{equation*}
    \end{minipage}
    \begin{minipage}{0.45\linewidth}
      \begin{equation*}
        F_i \in {\begin{tikzcd}[cramped, column sep=small, row sep=small]
            U \ar[dd] \\ \\ V
          \end{tikzcd}}
        \squareslash
        {\begin{tikzcd}[cramped, column sep=small, row sep=small]
            && E_i \ar[dd] \\ \\ B_i' \ar[rr] & & B_i
          \end{tikzcd}}
      \end{equation*}
    \end{minipage}
  \end{center}
  Let $R_{B_2}(E_2) \to E_2 \times_{B_2} E_2$ be a fibrewise relation.

  A \emph{structured $R_{B_2}(E_2)$-witness $H$ from $F_1$ to $F_2$ via
    $(e,b,b')$} is a witness $H$ that the top square of the following diagram
  commutes up to the fibrewise relation
  $[V, R_{B_2}(E_2)] \to [V, E_2 \times_{B_2} E_2] \cong [V, E_2] \times_{[V,
    B_2]} [V, E_2]$.
  \begin{equation*}
    \begin{tikzcd}[cramped]
      {\color{red0}[U, E_1] \times_{[U, B_1]} [V, B_1']}
      && {\color{red0}[U, E_2] \times_{[U, B_2]} [V, B_2']} \\
      & {\color{yellow0}[V, E_1]} && {\color{yellow0}[V, E_2]} \\
      {\color{darkblue0}[U, E_1] \times_{[U, B_1]} [V, B_1]}
      && {\color{darkblue0}[U, E_2] \times_{[U, B_2]} [V, B_2]} \\
      {\color{darkblue0}[V, B_1]} && {\color{darkblue0}[V, B_2]}
      \arrow[from=1-1, to=1-3]
      \arrow["{F_1}", dotted, from=1-1, to=2-2]
      \arrow[from=1-1, to=3-1, color=red0]
      \arrow["{F_2}", dotted, from=1-3, to=2-4]
      \arrow[from=1-3, to=3-3, color=red0]
      \arrow["H"{description}, phantom, from=2-2, to=1-3]
      \arrow[from=2-2, to=3-1, color=yellow0]
      \arrow[from=2-4, to=3-3, color=yellow0]
      \arrow[from=3-1, to=3-3]
      \arrow[from=3-1, to=4-1, color=darkblue0]
      \arrow[from=3-3, to=4-3, color=darkblue0]
      \arrow[from=4-1, to=4-3]
      \arrow[from=2-2, to=2-4, crossing over]
    \end{tikzcd}
  \end{equation*}

  If $(e,b,b')$ is the identity then we simply say \emph{$H$ is a structural
    $R_{B_2}(E_2)$-witness between $F_1$ and $F_2$}.
\end{definition}

With this formal definition in place, we show that in model categories, when the
fibrewise relation is taken to be the fibrewise path object and when the
structured lifting problem is either (trivial cofibration, fibration) or
(cofibration, trivial fibration) then there always exist structured lifting
solutions and moreover any two such structured lifting solutions are
structurally homotopic.
In other words, we will show in \Cref{thm:lift-htpy-unique} structured lifts
are structurally homotopically unique if they exist.
But for that we, first need the following technical lemma.

\begin{lemma}\label{lem:lift-htpy-unique-master}
  Let $\bC$ be locally cartesian closed and equipped with a model structure in
  which:
  \begin{itemize}
    \item $I$ is a good, but not necessarily very good, cylinder object for the
    terminal object.
    \item The cofibrations are preserved by products.
  \end{itemize}

  Further fix a map $E \to B$ and an object $A$ along with a
  factorisation of $[A,E] \to [A,B]$ into
  $\begin{tikzcd}[cramped, column sep=1em]
    {[A, E]} \ar[r, two heads,"{\sim}"{pos=0.2}] &
    {C} \ar[r] &
    {[A,B]}
  \end{tikzcd}$ where the first map is a trivial fibration.
  Take $P^I_{B}(E) \coloneqq [I,E] \times_{[I,B]} B$ to be the fibred $I$-path
  object of $E \to B$ constructed from $I$ so that one has the endpoint
  evaluation map
  $\ev_\partial \colon P^I_{B}(E) \to E \times_{B} E$.
  Then, for any $d \colon D \to C$ and pair of maps $f_0,f_1 \colon D \to [A,B]$
  over $C$ as below on the left we have a factorisation $H$ as below on the
  right.
  \begin{equation*}
    \begin{tikzcd}[cramped]
      D
      \ar[rr, "f_0"', yshift=-2]
      \ar[rr, "f_1", yshift=2]
      \ar[rd, "d"{description}]
      \ar[rdd, bend right=10, "{f}"']
      &&
      {[A,E]}
      \ar[ld, two heads, "{\sim}"{description}]
      \ar[ldd, bend left=10]
      \\
      & C \ar[d]
      \\
      & {[A,B]}
    \end{tikzcd}
    \rightsquigarrow
    \begin{tikzcd}[cramped, row sep=small, column sep=small]
      D \ar[rr, "{(f_0, f_1)}"] \ar[rd, dashed, "{\exists H}"'] && {[A, E \times_{B} E]}
      \\
      & {[A, P_{B}(E)]} \ar[ru, "{[A, \ev_\partial]}"']
    \end{tikzcd}
  \end{equation*}
\end{lemma}
\begin{proof}
  Because $\bC$ is cartesian closed, products preserve colimits, so in
  particular $1 \sqcup 1$ is preserved by products.
  Also, by assumption, cofibrations are preserved by products.
  Thus, the image of cofibration $1 \sqcup 1 \hookrightarrow I$ under the
  product with $D$ is the cofibration $D \sqcup D \hookrightarrow D \times I$.
  But also because
  $\begin{tikzcd}[cramped, column sep=small]
    {[A,E]} \ar[r, "{\sim}"{pos=0.2}, two heads] & C
  \end{tikzcd}$ is assumed to be a trivial fibration, we can solve the following
  lifting problem by some solution $\ell$ as on the left so we get a diagram on
  the right.
  \begin{equation*}
    \begin{tikzcd}
      D \sqcup D
      \ar[d, hook]
      \ar[rr, "{(f_0,f_1)}"]
      &
      &
      {[A,E]}
      \ar[d, two heads, "{\sim}"{description}]
      \ar[dd, bend left=30]
      \\
      D \times I
      \ar[r]
      \ar[urr, dashed, "{\ell}"{description}]
      &
      D
      \ar[r, "{d}"{description}]
      \ar[rd, "{f}"']
      &
      C
      \ar[d]
      \\
      && {[A, B]}
    \end{tikzcd}
    \rightsquigarrow
    \begin{tikzcd}
      D \times I
      \ar[d]
      \ar[r, "{\ell}"]
      &
      {[A, E]} \ar[d]
      \\
      D
      \ar[r, "{f}"']
      &
      {[A, B]}
    \end{tikzcd}
  \end{equation*}
  Because $D \times I \to D$ is the pushout-product of $0 \to D$ with $I \to 1$,
  Leibniz transpose of the right diagram above gives
  \begin{equation*}
    \begin{tikzcd}
      0 \ar[d] \ar[r]
      &
      {[A,E]}
      \ar[d]
      \\
      D
      \ar[r, "{(\ell^\dagger, f)}"']
      &
      {[I, [A,E]] \times_{[I, [A, B]]} [A,B]}
    \end{tikzcd}
  \end{equation*}
  It follows
  that one may choose $H \colon D \to [A, P^I_{B}(E)]$ as
  $H \coloneqq
  D \xrightarrow{(\ell^\dagger,f)}
  [I, [A,E]] \times_{[I, [A, B]]} [A, B] \cong
  [A, [I,E]] \times_{[A, [I, B]]} [A, B] \cong
  [A, [I,E] \times_{[I, B]} B] \cong
  [A, P_{B}^I(E)]$.

  We conclude by verifying that this choice of $H$ indeed factors
  $(f_0,f_1) \colon D \to [A, E \times_{B} E]$.
  This follows because for $\epsilon=0,1$ one has
  \begin{equation*}
    \begin{tikzcd}[cramped, row sep=small, column sep=small]
      D && {[A,B]} \\
      {[A, P_{\underline{B}}^I(B)]} \\
      {[A, [I,B] \times_{[I, \underline{B}]} \underline{B}]} \\
      {  [A, [I,B]] \times_{[A, [I, \underline{B}]]} [A, \underline{B}]} & {[I, [A,B]] \times_{[I, [A, \underline{B}]]} [A, \underline{B}]} & {[I, [A,B]]}
      \arrow["{f_\epsilon}", from=1-1, to=1-3]
      \arrow["H"', from=1-1, to=2-1]
      \arrow["{\ell^\dagger}"{description}, from=1-1, to=4-3]
      \arrow["\cong"', tail reversed, from=2-1, to=3-1]
      \arrow[from=3-1, to=4-1]
      \arrow[from=4-1, to=4-2]
      \arrow[from=4-2, to=4-3]
      \arrow["{\ev_\epsilon}"', from=4-3, to=1-3]
    \end{tikzcd}
  \end{equation*}
\end{proof}

With the tools in place, we now show that structured lifts are structurally
homotopically unique if they exist.

\begin{theorem}\label{thm:lift-htpy-unique}
  Let $\bC$ be locally cartesian closed and equipped with a model structure in
  which:
  \begin{itemize}
    %
    %
    \item $I$ is a good, but not necessarily very good, cylinder object for the
    terminal object.
    \item Cofibrations are preserved by products.
    \item The pullback-power of (trivial cofibration, fibration)- and
    (cofibration, trivial fibration)-pairs are trivial fibrations.
  \end{itemize}
  Import the data $U \to V$ and $(e,b,b')$ and $F_1,F_2$ from
  \Cref{def:struct-lift-rel} where the fibrewise relation
  $R_{B_2}(E_2) \to E_2 \times_{B_2} E_2$ is taken to be the boundary evaluation
  map of the fibred $I$-path object
  $P_{B_2}^I(E_2) = [B_2 \times I, E_2]_{B_2} \to E_2 \times_{B_2} E_2$.

  Assume that $(U \to V, E_2 \to B_2)$ is either a (trivial cofibration,
  fibration) or (cofibration, trivial fibration) pair.
  Then, there exists a structured $P_{B_2}^I(E_2)$-witness $H$ from $F_1$ to
  $F_2$ via $(e,b,b')$.
\end{theorem}
\begin{proof}
  We must show that the top square below commutes up to the fibrewise homotopy
  relation $P_{B_2}^I(E_2)$ with some witness $H$.
  \begin{equation*}
    \begin{tikzcd}[cramped]
      {[U, E_1] \times_{[U, B_1]} [A, B_1']} && {[U, E_2] \times_{[U, B_2]} [A, B_2']} \\
      & {[A, E_1]} && {[A, E_2]} \\
      {[U, E_1] \times_{[U, B_1]} [A, B_1]} && {[U, E_2] \times_{[U, B_2]} [A, B_2]} \\
      {[A, B_1]} && {[A, B_2]}
      \arrow[from=1-1, to=1-3]
      \arrow["{F_1}", dashed, from=1-1, to=2-2]
      \arrow[from=1-1, to=3-1]
      \arrow["{F_2}", dashed, from=1-3, to=2-4]
      \arrow[from=1-3, to=3-3]
      \arrow["H"{description}, phantom, from=2-2, to=1-3]
      \arrow[from=2-2, to=3-1]
      \arrow[from=2-4, to=3-3]
      \arrow[from=3-1, to=3-3]
      \arrow[from=3-1, to=4-1]
      \arrow[from=3-3, to=4-3]
      \arrow[from=4-1, to=4-3]
      \arrow[from=2-2, to=2-4, crossing over]
    \end{tikzcd}
  \end{equation*}
  But the condition that $(U \to V, E_2 \to B_2)$ is either a (trivial
  cofibration, fibration) or (cofibration, trivial fibration) pair ensure that
  $\begin{tikzcd}[cramped, column sep=1em] {[V, E_2]} \ar[r, two heads,
    "{\sim}"{pos=0.2}] & {[U, E_2] \times_{[U,B_2]} [V,E_2]}
  \end{tikzcd}$ is a trivial fibration, so the result follows by
  \Cref{lem:lift-htpy-unique-master}.
\end{proof}


\printbibliography

\end{document}